\documentclass[12 pt]{article}
 
 \usepackage{a4wide}
 \usepackage{natbib}
 
\usepackage{amsmath,amsfonts,amssymb,amsthm,booktabs,color,epsfig,graphicx,hyperref,url,lineno}
 
 \usepackage{booktabs}
 
\newtheorem{theorem}{Theorem}
\newtheorem{assumption}{Assumption}
\newtheorem{lemma}{Lemma}
\theoremstyle{remark}
\newtheorem{remark}{Remark}
\newtheorem{example}{Example}

\usepackage{ulem}

\usepackage{enumitem}

\newcommand{\Rr}{\ensuremath{\mathbb{R}}}

\newcommand{\R}{\mathbb{R}}
\newcommand{\Lcal}{\ensuremath{{\cal L}}}

\newcommand{\Ncal}{\ensuremath{{\cal N}}}

\renewcommand{\P}{\mathsf{P}}

\newcommand{\tv}{\mathbf{v}}

\newcommand{\ty}{\mathbf{y}}
\newcommand{\tz}{\mathbf{z}}
\newcommand{\txi}{\boldsymbol \xi}
\newcommand{\teps}{\bs{\eps}}
\newcommand{\teta}{\bs{\eta}}
\newcommand{\eps}{\varepsilon}
\newcommand{\var}{\mathrm{Var}\,}
\newcommand{\cov}{\mathrm{Cov}\,}
\newcommand{\E}{\mathrm{E}\,}

\newcommand{\I}{\mathrm{I}}

\newcommand{\tE}{\mathbb{E}}
\newcommand{\tD}{\mathbf{D}}

\newcommand{\tR}{\mathbf{R}}
\newcommand{\tX}{\boldsymbol{X}}
\newcommand{\tY}{\boldsymbol{Y}}
\newcommand{\tZ}{\boldsymbol{Z}}
\newcommand{\tB}{\mathbf{B}}
\newcommand{\tu}{\mathbf{u}}
\newcommand{\tc}{\mathbf{c}}
\newcommand{\tb}{\mathbf{b}}

\newcommand{\tP}{\boldsymbol{P}}

\newcommand{\tmu}{\boldsymbol{\mu}}
\newcommand{\Dto}{\stackrel{D}{\to}}
\newcommand{\Pto}{\stackrel{P}{\to}}

\newcommand{\ta}{\mathbf a}
\newcommand{\bh}{\widehat{\mathbf b}}
\newcommand{\ah}{\widehat{\mathbf a}}

\newcommand{\what}[1]{\widehat{#1}}

\renewcommand{\th}{(\ta_{\tau,\tu}^\top,\tc_{\tau,\tu}^\top)^\top}

\newcommand{\tct}{{\tc}_{\tau,\tu}}

\newcommand{\tU}{\mathbf{U}}
\newcommand{\tS}{\mathbf{S}}

\newcommand{\tr}{\mathbf{r}}

\newcommand{\td}{\mathbf{d}}

\newcommand{\tW}{\mathbf W}

\newcommand{\dd}{\,\mathrm{d}}

\newcommand{\Gu}{\mathbf{\Gamma}_{\tu}}

\newcommand{\tA}{\mathbf{A}}

\newcommand{\tOmega}{\mathbf{\Omega}}
\newcommand{\tC}{\mathbf{C}}
\newcommand{\tI}{\mathbf{I}}

\newcommand{\ttY}{\mathbb{Y}}
\newcommand{\ttX}{\mathbb{X}}
\newcommand{\ttI}{\mathbf{I}}
\newcommand{\ttH}{\mathbf{H}}
\newcommand{\ttM}{\mathbf{M}}
\newcommand{\ttG}{\mathbf{G}}

\newcommand{\ttA}{\mathbf{A}}
\newcommand{\ttB}{\mathbf{B}}

\newcommand{\fuX}{f_{\tu^\top\tY|\tX}}

\newcommand{\fuXG}{f_{\tu^\top\tY|\tX,\Gu^\top\tY}}

\newcommand{\tSigma}{\boldsymbol{\Sigma}}
\newcommand{\tx}{\boldsymbol{x}}
\newcommand{\tQ}{\mathbf{Q}}
\newcommand{\tV}{\mathbf{V}}

\newcommand{\bs}[1]{{\boldsymbol{#1}}}

\begin{document}

\begin{center}

\textbf{\Large Testing axial symmetry in multivariate location-scale \\[1ex] linear regression}
\vskip4ex

\v{S}\'{a}rka Hudecov\'{a}$^1$ and Miroslav \v{S}iman$^2$

 \end{center}

\vskip4ex

{\small
\noindent $^1$
Department of Probability and Statistics,
Faculty of Mathematics and Physics,  Charles University, Czech Republic\newline
Email address: \url{hudecova@karlin.mff.cuni.cz} \\[1ex]
$^2$ The Czech Academy of Sciences, Institute of Information Theory and Automation, Czech Republic
}

\paragraph{Abstract.}
The article deals with the problem of  testing conditional axial symmetry within a~multivariate linear  heteroscedastic regression framework. 
A new test based on integrated rank scores is introduced and its asymptotic distribution is derived. The proposed method extends a similar procedure developed for multivariate data to the regression setting. The test may also be employed to assess specific hypotheses concerning 
distributional properties of the error term. Its performance and application is illustrated in a small  simulation study and with real economic data. The article also contains a few theoretical results regarding axial symmetry that may be of independent interest.

\section{Introduction}

Many multivariate techniques assume some form of symmetry, and violations of the assumption can lead to invalid conclusions. Also the optimal choice and efficiency of statistical methods can be influenced by a symmetry of the studied data distribution. Furthermore, symmetry plays an essential role in nature. Therefore, testing hypotheses about various multivariate symmetries is an important area of study; see, for example, \citep{He95, Di99, HM12, De23, Ch24}.

Testing axial symmetry in multivariate spaces has recently emerged as an interesting research topic, as it can also be applied to testing certain hypotheses regarding independence, exchangeability, or goodness-of-fit, see \cite{Hu19}. Moreover, some economic theories propose symmetric or specific relationships between variables (e.g., balanced trade, equal utility in consumption) that can be  translated as a  symmetry of the underlying distribution with respect to a specific axis.
This problem is particularly appealing in the context of multivariate regression location-scale models.

Let $m\geq 2$ and $p\geq 1$ be integers, let $\tY$ be an $m$-dimensional real random vector of outcomes
and let $\tX$ be a $p$-dimensional real random vector of regressors. These vectors will be assumed to satisfy the following linear location-scale heteroscedastic regression model:
\begin{equation}\label{RegMod1}
\tY=\ttB\tX + (\td^\top\tX)\teps,
\end{equation}
where $\ttB$ is a real $m\times p$ constant matrix, $\td\in\R^p$ is a nonzero real constant vector 
such that $\td^\top \tX\ne 0$, and $\teps$ is an $m$-dimensional real random vector 
independent of $\tX$, with $\E \teps =\bs{0}$. 
As usual, the first component of $\tX$ is assumed to be $1$. The model reduces to the classic homoscedastic multiple linear regression if 
$\td=(1,0,\dots,0)^\top$. To ensure the indentifiability of model parameters, we further require the first component of $\td$ equal to~1, i.e., 
$\td = (1, d_2, \dots, d_p )^\top$.

If the variance matrix of $\teps$ is finite, then the first two conditional moments of $\tY$ from \eqref{RegMod1} given $\tX$ are $\E[\tY|\tX] = \ttB \tX$ and $\var[\tY|\tX] = (\td^\top\tX)^2 \var \teps$. Nevertheless, one may sometimes desire to have additional information about the whole conditional distribution $\Lcal(\tY|\tX)$. For example, some approaches to Value-at-Risk or to the construction of prediction sets assume that the conditional distribution is multivariate normal or elliptical \citep{La18}. Econometricians may want to test whether the conditional distribution is exchangeable \citep{Mc91}. In general, the assumption of conditional symmetry proved useful for statistical inference, see, e.g., \cite{HM12}, \cite{Ch22}, \cite{Ch24} and references given there.  In this article, we are interested in testing the axial symmetry around a line in a given direction. This concept can be used for testing various statistical hypotheses  \citep{Hu19, Hu20} and relates to other distributional properties and concepts of symmetry.

In the non-regression case, there are already quite a few nonparametric tests of the same or similar hypotheses for special axes and/or in the bivariate case when axial and halfspace symmetry coincide; see, e.g., \cite{Ho71}, \cite{Mo08}, \cite{Ra12}, \cite{Ri21} and references therein for some examples, though usually assuming a fixed axis of symmetry rather than its direction. Unfortunately, the nonparametric tests of axial symmetry usable in spaces beyond dimension two and working with all possible axial directions are only recent and few; see \cite{Ka19}, \cite{Hu19, Hu20, Hu21}, \cite{Si24}, perhaps also \cite{Hu25}. And only those of \cite{Hu19} and \cite{Si24} can be used in the linear regression setup where their performance is far from satisfactory. As even the best tests of \cite{Si24} are recommended only for very large datasets of small-dimensional observations when other options are not available, the only real competitor to the test presented here is described in \cite{Hu19}.   
This article generalizes the test of \cite{Hu21} to the linear regression setup. The resulting test inherits many desirable properties from the purely multivariate case,   
as power behavior, simplicity, and flexibility. 

We would also like to mention for the sake of completeness that the literature also describes a few tests of somewhat similar hypotheses, for example the tests regarding rotational symmetry of directional data (\cite{Pa20c}, \cite{Pa20d}, \cite{Cu20}, \cite{Wa21}, \cite{Ve22}), the tests of conditional central symmetry in case of zero trend (\cite{Su06}, \cite{Ch22}, \cite{Ch24}), the tests for the first eigenvector of a scatter matrix (\cite{Ha10b}, \cite{Pa20a}, \cite{Pa20b}, \cite{Pa23}) or the tests for eigenspaces of covariance matrices \citep{Ty81,Si23}.
   
Section \ref{DN} formally introduces the concept of axial symmetry. The proposed test statistic is explained in Section \ref{TH}, while its asymptotic distribution is provided in Section~\ref{AS}. Section \ref{Demo} confirms the theoretical results with a small simulation study, and Section \ref{Real} illustrates the proposed methodology with an application to real economic data.  Concluding comments are  collected in Section \ref{CR}. 
Some auxiliary results about axial symmetry and its relation to the other concepts of distributional symmetry are provided in \ref{ap1}. 
All proofs are postponed to \ref{ap1B}, and all figures complementing the practical part of the paper are collected in \ref{ap2}.

\section{Axial symmetry in the regression setup}
\label{DN}

Let $\tu\in\R^m$ be a known fixed vector such that  $\|\tu\|=1$.  Let $\Gu$ be any matrix such that the matrix $\tQ=(\tu|\Gu)$ is orthogonal, i.e., $\tQ\tQ^\top = \tQ^\top\tQ = \tI$, where $\tI$ is the identity matrix. We also write $\boldsymbol{1}$ for the vector of all 1's, $\I[\cdot]$ for the indicator function, and $\Lcal(\tZ)$ for the distribution of a~random vector $\tZ$.

An $m$-dimensional stochastic vector $\tZ$ with finite expectation is said to be axially symmetric around an axis with direction specified by the unit vector $\tu \in \R^m$ if the distribution of $\tZ - \E\tZ$ is the same as the distribution of $\tR_{\tu}(\tZ-\E\tZ)$, where  $\tR_\tu = 2\tu\tu^\top - \tI$ is an orthogonal reflection matrix such that $\tR_\tu \tu = \tu$ and $\tR_\tu \tv = -\tv$ for all vectors $\tv\in\R^m$ orthogonal to $\tu$. 
The axial symmetry of a distribution with infinite expectation could be defined by replacing $\E\tZ$ with any point lying on the axis of symmetry in direction $\tu$. 
In particular, $\tZ$ with density $f$ is axially symmetric around the axis in direction $\tu$ and passing through the origin if and only if 
$f(\tz)=f(\tR_\tu \tz)$ for all $\tz\in\R^m$. 

According to Lemma \ref{lemEqiv1} in \ref{ap1}, $\tZ$ is axially symmetric around an axis with direction $\tu$, if and only if $(\tu|\Gu)^\top(\tZ-\E\tZ)$ has the same distribution as $(\tu|-\Gu)^\top (\tZ-\E\tZ)$. \ref{ap1} also contains other lemmas elucidating axial symmetry and its connections to the  eigenvectors of the scatter matrix and other concepts of symmetry.

For an outcome  $m$-dimensional  random vector $\tY$ and a $p$-dimensional regressor vector $\tX$, we are interested in testing the  null hypothesis that the conditional distribution of $\tY$ given $\tX$ is symmetric around  a~line in direction
 $\tu$,
\begin{equation*}
H_0(\tu): \ \Lcal(\tY|\tX) \text{ is axially symmetric around a line in direction } \tu. 
\end{equation*}
If  vector $(\tY^\top,\tX^\top)^\top$ satisfies model \eqref{RegMod1},  then the conditional distribution of  $\tY$ given $\tX$ is obtained as a special affine  transformation of the distribution of $\teps$. The null hypothesis $H_0(\tu)$ then holds if and only if the distribution of $\teps$ is axially symmetric around a line in direction 
$\tu$. Mathematically,  
\[
H_0(\tu): \ \Lcal(\teps) = \Lcal(\tR_\tu \teps).
\]

\medskip

Let $(\tY_i^\top,\tX_i^\top)^\top$, $i=1,\dots,n$, be a random sample of $n$ independent random vectors from \eqref{RegMod1}. 
Write $\ttY=(\tY_1,\dots,\tY_n)^\top$ and $\ttX = (\tX_1,\dots,\tX_n)^\top$ for the $n\times m$ and $n\times p$ data matrices. Then the corresponding regression model can be written as
\begin{equation}\label{RegMod2}
\ttY = \ttX \tB^\top+\tD \tE,
\end{equation}
where $\tE=(\teps_1,\dots,\teps_n)^\top$ and $\tD=\mathrm{diag}\{\td^\top\tX_1,\dots,\td^\top \tX_n\}$. 
If the random errors $\teps_1,\dots,\teps_n$ were observed, then $H_0(\tu)$ could be tested via the tools already available for testing the axial symmetry of multivariate distributions.
However, as the errors are not observed,  it is necessary to construct the test for $H_0(\tu)$ only by means of the data matrices $\ttY$ and $\ttX$.

\section{Testing by means of integrated rank scores}\label{TH}

There exists  a useful connection between $H_0(\tu)$ and the directional quantile regression introduced in the article \cite{Ha10a}. 
There, for a  quantile level $\tau\in(0,1)$, the $(\tau,\tu)$-dependent directional regression quantile coefficient vector ${\th}$ is defined as the solution to the minimization problem
\begin{equation}\label{optim}
\min_{(\ta^\top,\tc^\top)^\top\in \Rr^{m+p-1}} \E \rho_\tau \Bigl(\tu^\top\tY-\ta^\top\tX-\tc^\top\Gu^\top\tY\Bigr),
\end{equation}
where $\rho_\tau (t)= t (\tau - \I[t<0]) = \max\{(\tau-1)t,\tau t\}$ is the quantile check function. The null hypothesis $H_0(\tu)$ implies that $\tct=\bs{0}$ for all $\tau\in(0,1)$, see Lemma~\ref{lemHS} in  \ref{ap1B}, and the reverse is true for
elliptically distributed $\teps$, see \cite{Hu19}. Testing the hypothesis
\begin{equation*}
\widetilde{H}_0(\tu):\ \tct=\bs{0}\ \ \text{for all}\  \tau\in(0,1)
\end{equation*}
is closely related to the quantile regression problems considered in \cite{Gu92}, \cite{Gu93} and \cite{Ko99}. Unfortunately, the presence of stochastic and response-dependent regressors complicates the matter and the results contained in those articles cannot be applied directly. 
 
In the non-regression case, where $\ttX=\boldsymbol{1}$, \cite{Hu21} showed that it is beneficial to test axial symmetry by means of the integrated rank scores of \cite{Gu92}. The integrated rank score vector  $\bh=(\what{b}_{1},\dots,\what{b}_{n})^\top$, 
\begin{equation*}\label{bi}
\what{b}_{i}=-\int_{0}^1\phi(t) \dd \what{a}_{i}(t), \ i=1,\dots,n,
\end{equation*}
is defined for any integrable score function $\phi:[0,1]\to \R$ by means of the rank score vector $\ah(\tau) = \big(\what{a}_1(\tau),\dots,\what{a}_n(\tau)\big)^\top$,
\begin{equation}\label{eq:DQ}
\ah(\tau) = \mathrm{argmax}_{\ta\in[0,1]^n}\{\tu^\top\ttY^\top\ta :\ \ttX^\top {\ta}=(1-\tau)\ttX^\top\mathbf{1} \},\ \tau \in [0,1].
\end{equation}
In what follows, we assume that $\phi$ satisfies the following condition:

\begin{assumption}\label{as:phi}
Function $\phi:[0,1]\to \R$ is a non-decreasing real function constant outside $[\varepsilon,1-\varepsilon]$ for some $\varepsilon\in(0,\frac12)$ and non-constant on $[\varepsilon,1-\varepsilon]$.
\end{assumption}

If Assumption \ref{as:phi} holds, then $\int_0^1 \phi^k(t)\dd t<\infty$ for any integer $k>0$. Therefore, $\phi$ is square integrable with finite mean $\bar{\phi}:=\int_0^1 \phi(t)dt$,
and the integration by parts leads to 
\[
\what{b}_{i} = [-\phi(t) \what{a}_i(t)]_0^1 + \int_{0}^1 \what{a}_i(t)\dd\phi(t) =\phi(0) + \int_{0}^1 \what{a}_i(t)\dd\phi(t),
\] 
because $\what{a}_i(0)=1$ and $\what{a}_i(1)=0$.
Possible examples of $\phi(t)$ satisfying Assumption~\ref{as:phi}  include $\phi(t)=\mathrm{sign}(t-0.5)$ (leading to the sign scores), 
the trimmed Wilcoxon score function such that  $\phi(t)=(t-0.5)$ for $t\in[\eps,1-\eps]$ 
 (leading to the trimmed Wilcoxon scores), and $\phi$ equal to the trimmed quantile function of the standard normal distribution, i.e., $\phi(t) =\Phi^{-1}(t)$ for $t\in[\eps,1-\eps]$,  which leads to the trimmed normal (or, van der Waerden) scores. 
Note that all the three illustrative score functions lead to $\bar{\phi}=0$ and $\phi^2(t)=\phi^2(1-t)$, $t\in [0,1]$. The trimming outside $[\varepsilon, 1-\varepsilon]$ can always be done for any unbounded score function for a very small $\varepsilon$, and as such, it has no practical effect for a fixed finite  sample size.

Denote as $\ttM_{\ttX} = \ttX(\ttX^\top\ttX)^{-1}\ttX^\top$ the projection matrix onto the regression space $\ttX$. 
We generalize the statistic of \cite{Hu21} to the regression case and consider it in the form 
\begin{equation}\label{Sn}
\what{\tS}_n=\frac{1}{\sqrt{n}}\Gu^\top {\ttY}^\top(\ttI - \ttM_{\ttX})\bh.
\end{equation}
  It follows from the connection between $H_0$ and $\widetilde{H}_0$ that $\what{\tS}_n$ can be used for testing the null hypothesis of axial symmetry.

The computation of $\what{\tS}_n$ is easy because the integrated rank score vector $\bh$ can be obtained by means of the tools available for quantile regression with responses $\tu^\top \tY_i$ and regressors $\tX_i$, $i=1,\dots,n$. 
For example, this can be achieved  in R \citep{R} using the package \texttt{quantreg} \citep{Ko15} and functions \texttt{qr} and \texttt{ranks}.

\section{Asymptotics}\label{AS}

%

We derive the asymptotic results regarding $\what{\tS}_n$ under the following assumption:
 
\begin{assumption}\label{as:XY}
Let $(\tY^\top,\tX^\top)^\top$ satisfy model \eqref{RegMod1} with $\E \varepsilon=\bs{0}$ and $\tX=(1,\widetilde{\tX}^\top)^\top$. 
Let both $(\tY^\top,\widetilde{\tX}^\top)^\top$ and $\teps$ have a continuous distribution with a density that is bounded, continuous and positive in the interior 
of a connected support. Finally, let $\E\|\bs{Y}\|^{2+\delta}<\infty$ and $\E\|\bs{X}\|^{2+\delta}<\infty$ for some $\delta>0$.
\end{assumption}

For $i=1,\dots,n$, define 
\begin{align}\label{eq:WU}
W_i&= (\td^\top\tX_i) \tu^\top \teps_i = \tu^\top\tY_i - \tu^\top\ttB\tX_i, \ \ \text{and}\\
\tU_i&=(\td^\top\tX_i )\Gu^\top\teps_i = \Gu^\top\tY_i - \Gu^\top\ttB\tX_i.\notag
\end{align}

\begin{theorem} \label{prop:p1} 
Let $H_0(\tu)$ hold, Assumptions \ref{as:phi} and \ref{as:XY} be satisfied  and  $\td^\top \tX_1>0$. 
Then 
\begin{equation}
\what{\tS}_n \Dto \mathsf{N}(\mathbf{0},\tSigma) \ \text{as} \ n\to\infty,
\end{equation} 
where 
 \begin{equation}\label{varT2}
\tSigma = \E \left\{ \tU_1 \tU_1^\top  \left[ \phi\big(F_{ \tu^\top \teps}(\tu^\top \teps_1)\big)-\overline{\phi}\right]^2\right\},
\end{equation}
where $F_{ \tu^\top \teps}$ is the  cumulative distribution function of $\tu^\top \teps_1$.
\end{theorem}

The formula for $\tSigma$ can be further simplified in some special cases:

\begin{theorem}\label{th2} Consider the same assumptions as in Theorem~\ref{prop:p1}.
If either\\
(i) $\phi$ is the sign score function, $\phi(t) = \mathrm{sgn}(t-1/2)$, or\\
(ii)  $\tU_1 $ and $\tu^\top\teps_1$ are conditionally independent given $\tX_1$, \\
then
\begin{equation}\label{sigma0}
\tSigma=\left( \E  \tU_1 \tU_1^\top \right)\cdot \sigma^2_{\phi},  
\end{equation}
where 
\[
\sigma^2_{\phi} = \int_{0}^1 \left[ \phi(v)-\overline{\phi}\right]^2 \dd v.
\]
\end{theorem}

As a consequence of Theorem~\ref{prop:p1}, if $\widehat{\tSigma}$ is a consistent estimator of $\tSigma$ and $H_0(\tu)$ holds, then the null hypothesis can be tested by means of the test statistic
\begin{equation}\label{test}
{T}_n:=\what{\tS}_n^\top\what{\tSigma}^{-1} \what{\tS}_n
\end{equation}
that converges to the $\chi^2_{m-1}$ distribution as $n\to\infty$. The null hypothesis $H_0(\tu)$ should be rejected for 
large values of $T_n$.

Note that under the conditions of Theorem~\ref{th2}, the test statistic  $T_n$  has the same limiting distribution as the rank score $\chi^2$-test statistic of (2.9) in \cite{Gu93}, obtained as if regressors $\Gu^\top \tY_i$ and $\tX_i$ in the 
quantile regression model for $\tu^\top \tY_i$ were deterministic.
It follows from Lemma~\ref{lem:Norm}    that $\tU_1$ and $\tu^\top\teps_1$ are conditionally independent given $\tX_1$ in the important special case when the distribution of $\teps$ is multivariate normal.

\begin{remark} 
Let $\widehat{\tB}$ be a consistent estimator of $\ttB$, for instance the ordinary least squares  (OLS) estimator $\widehat{\tB} = \ttY^\top\ttX (\ttX^\top \ttX)^{-1}$. 
For the situation described in Theorem~\ref{th2}, 
an estimator of $\tSigma$ can be constructed easily as
\begin{equation}\label{eq:Sigma}
\widehat{\tSigma} = \sigma^2_{\phi} \cdot \frac{1}{n} \sum_{i=1}^n \widehat{\tU}_i \widehat{\tU}_i^\top. 
\end{equation}
where $\widehat{\tU}_i$ are defined as in \eqref{eq:WU}, but with $\tB$ replaced by $\widehat{\tB}$.  Observe that if $n\geq m$, then $\widehat{\tSigma}$ is invertible with probability 1.

For a more general setup with formula \eqref{varT2},  let  $\widehat{\td}$ be  a consistent estimator of  $\td$ and  define $\widehat{W}_i$  as in \eqref{eq:WU}, but with $\tB$ replaced by $\widehat{\tB}$. 
Set
\[
\widehat{F}_e(y) = \frac{1}{n}\sum_{i=1}^n \mathrm{I}[e_i\leq y]
\]
for
\[
e_i = \frac{\tu^\top\tY_i - \tu^\top\widehat{\tB}\tX_i}{\widehat{\td}^\top \tX_i} = \frac{\widehat{W_i}}{\widehat{\td}^\top \tX_i} ,\ \ i=1,\dots,n,
\]
and define the estimator of $\tSigma$ as
\begin{equation}\label{eq:Sigma.est}
\widehat{\tSigma} = \frac{1}{n}\sum_{i=1}^n \widehat{\tU}_i \widehat{\tU}_i^\top \left[\phi\bigl(\widehat{F}_e(e_i)\bigr) - \overline{\phi} \right]^2.
\end{equation}

Recall that the homoscedastic case, where $\tY_i = \tB \tX_i + \teps_i$, $i=1,\dots,n$, is obtained for $\td=(1,0,\dots,0)^\top$. In that case, one just takes $\widehat{\td}^\top \tX_i = 1$ and
$e_i=\widehat{W}_i$ for all $i=1,\dots,n$.  Consequently, $\widehat{F}_e$ is then the sample cumulative distribution function of $\widehat{W}_1,\dots,\widehat{W}_n$. 
\end{remark}

\begin{remark} 
The test statistic $T_n$ from \eqref{test} can be expressed as
 \[
 T_n = \frac{1}{n}\, \widehat{\tb}^\top (\tI - \ttM_{\ttX})\ttY \Gu \what{\tSigma}^{-1} \Gu^\top \ttY^\top(\tI - \ttM_{\ttX}) \widehat{\tb}.
 \]
This implies that $T_n = T_n (\ttY,\ttX)$ enjoys some interesting invariant properties. 

\smallskip

(i) \emph{Independence on the choice of $\Gu$.}  
The test statistic
$T_n$ is invariant with respect to the choice of $\Gu$,
provided that the estimator  $\what{\tSigma} $ is constructed as  in 
 \eqref{eq:Sigma} or \eqref{eq:Sigma.est}. 
 
Indeed, the estimators in \eqref{eq:Sigma} and \eqref{eq:Sigma.est}  can be written as $\what{\tSigma}= \Gu^\top \tC_n \Gu$ for some matrix $\tC_n$ that does not depend on $\Gu$. 
 If $\widetilde{\Gu}\ne \Gu$ is another complement of the vector $\tu$ to an orthogonal matrix, then both $\Gu$ and $\widetilde{\Gu}$ are bases of the space orthogonal to $\tu$ and there exists an orthogonal matrix $\tA$ such that $\widetilde{\Gu} = \Gu \tA$.  It then easily follows that $ \widetilde{\Gu}(\widetilde{\Gu}^\top \tC_n \widetilde{\Gu})^{-1}\widetilde{\Gu}^{\top}= \Gu(\Gu^\top \tC_n \Gu)^{-1}\Gu^{\top}$. 
Therefore, $T_n$ does not depend on the choice of $\Gu$, because $\widehat{\tb}$ does not depend on $\Gu$ either.
 
 \smallskip
 
(ii) \emph{Shift invariance.} Let $\ttA\in \Rr^{p\times m}$ be a matrix.  The test statistic $\widehat{\tS}_n = \widehat{\tS}_n(\ttY,\ttX)$ from \eqref{Sn} is invariant with respect to the transformation $\bs{Y}\mapsto \bs{Y}+\ttA^\top\bs{X}$. Since the first component of $\tX$ is 1, this involves also the special case of a deterministic shift  $\bs{Y}\mapsto \bs{Y}+\bs{s}$ for  a vector $\bs{s}\in \Rr^m$ that is obtained for $\ttA^\top = (\bs{s} \, | \bs{0}_{m \times p-1})$. If the estimator $\what{\tSigma} $ satisfies $\what{\tSigma}(\ttY+\ttX\ttA,\ttX) = \what{\tSigma}(\ttY,\ttX)$, then also 
$T_n$ 
is invariant to this transformation.

Indeed, if $\widetilde{\tY} = \tY +\ttA^\top\bs{X}$, then  $\widetilde{\tY}$ satisfies model \eqref{RegMod1} with $\widetilde{\tB}=\tB+\ttA^\top$ and $\widetilde{\ttY} = \ttY + \ttX \ttA$. 
Furthermore, in \eqref{eq:DQ}, $\tu^\top \widetilde{\ttY}^\top \ta = \tu^\top\ttY^\top \ta + \tu^\top\tA^\top \cdot (1-\tau)\ttX^\top \boldsymbol{1}$, where the second summand does not depend on~$\ta$. Therefore, the regression scores and integrated regression scores remain unchanged.  
Consequently, it follows from \eqref{Sn} and the definition of $\ttM_{\ttX}$ that $\widehat{\tS}_n(\widetilde{\ttY},\ttX) = \widehat{\tS}_n(\ttY,\ttX) $. 

The assumption  $\what{\tSigma}(\ttY+\ttX\ttA,\ttX) = \what{\tSigma}(\ttY,\ttX)$ is always satisfied by \eqref{eq:Sigma} and \eqref{eq:Sigma.est} provided that $\widehat{\tB}$  is shift equivariant (which is the case for the OLS  estimator) and $\widehat{\td}$ is shift invariant.

\smallskip

(iii) \emph{Rotation invariance.} Let $\ttA$ be an orthogonal $m\times m$ matrix and consider
 the rotational transformation $\tY \mapsto  \ttA \tY$. If the estimator $\what{\tSigma}$ is constructed as in  \eqref{eq:Sigma} or \eqref{eq:Sigma.est}, with estimators satisfying $\widehat{\tB}(\ttY \tA^\top,\ttX) = \tA \widehat{\tB}(\ttY,\ttX)$ (which is the case for the OLS estimator) and $\widehat{\td}(\ttY \tA^\top,\ttX) = \widehat{\td}(\ttY,\ttX)$, then $T_n$ is invariant to this transformation. 
 
Indeed, set $\widetilde{\tY} =  \ttA \tY= (\tA \tB) \tX + (\td^\top \tX) \tA \teps$. Then $\widetilde{\tY}$ satisfies model \eqref{RegMod1} with $\widetilde{\tB} =  \tA \tB$ and error term $\widetilde{\teps} = \tA \teps$. It follows from the definition of axial symmetry that $\teps$ is axially symmetric around $\tu$ if and only if $\widetilde{\teps}$ is axially symmetric around $\widetilde{\tu} = \tA \tu$. Let $\widetilde{T}_n$ be the test statistic constructed from $\widetilde{\ttY} = \ttY \tA^\top$ and $\ttX$ for testing $H_0(\widetilde{\tu})$ for $\widetilde{\tu}=\tA \tu$. Note that in \eqref{eq:DQ}, $\widetilde{\tu}^\top \widetilde{\ttY}^\top \ta = \tu^\top\ttY^\top \ta$, and, therefore, 
the rank score vector corresponding to the triple $(\widetilde{\tu},\widetilde{\ttY},\ttX)$ is the same as the rank score vector corresponding to $(\tu,\ttY,\ttX)$.  
We can consider $\widetilde{\Gu} = \tA \Gu$ in the definition of $\widetilde{T}_n$ thanks to (i) above. Then $\widetilde{\ttY} \widetilde{\Gu} = 
\ttY \tA^\top \tA \Gu = \ttY\Gu$, and it follows from \eqref{Sn} that $\widehat{\tS}_n(\widetilde{\ttY},\ttX, \widetilde{\tu}) = \widehat{\tS}_n( \ttY,\ttX,\tu)$. The assumptions on $\widehat{\tSigma}$ imply that also  $\widetilde{T}_n = T_n$.

\smallskip

(iv) \emph{Invariance with respect to a linear transformation in $\tX$.}  Let $\tC\in \Rr^{p\times p}$ be a regular matrix and consider the transformation $\tX \mapsto\tC \bs{X}$.
If the estimator $\what{\tSigma}$ satisfies  $\what{\tSigma}(\ttY,\ttX\tC^\top) = \what{\tSigma}(\ttY,\ttX)$,  
then $T_n$ is  invariant with respect to this transformation. 
This assumption 
is always satisfied by $\what{\tSigma}$ in \eqref{eq:Sigma} and \eqref{eq:Sigma.est} provided that $\widehat{\tB}$
satisfies $\widehat{\tB}(\ttY, \ttX \tC^\top) =\widehat{\tB}(\ttY, \ttX)  \tC^{-1}$ (which is the case for the OLS estimator) and $\widehat{\td}$ satisfies $\widehat{\td}(\ttY, \ttX \tC^\top) = (\tC^\top)^{-1} \widehat{\td}(\ttY, \ttX )$.

Indeed, set $\widetilde{\tX} = \tC \tX$. 
Then
$\tY = (\tB\tC^{-1}) \widetilde{\tX} +  [(\tC^{\top})^{-1}\td]^{\top} \widetilde{\tX} \teps$, which implies that the couple $(\tY,\widetilde{\tX})$ satisfies model \eqref{RegMod1} with the same error term $\teps$. Note that
 $\widetilde{\ttX} = \ttX \tC^\top$ and $\ttM_{\ttX} =  \ttM_{\widetilde{\ttX}}$. 
Furthermore, it follows from \eqref{eq:DQ} that $\widehat{\tb}(\ttY,\widetilde{\ttX}) = \widehat{\tb}(\ttY,\ttX)$.
Hence, $\widehat{\tS}_n (\ttY,\widetilde{\ttX}) = \widehat{\tS}_n (\ttY,\ttX)$.  Finally, the invariance assumptions on $\what{\tSigma}$ imply also the invariance of $T_n$.

\end{remark}
\section{Simulation Study}
\label{Demo}

This section explores the small sample properties of the test based on the test statistic $T_n$. 
The simulation study
is designed to be representative:  It uses random samples of size $n = 250, 500,$ or $1\;000$, response dimension $m = 2, 5,$ or $10$, regressor dimension $p= 1, 3, 5,$ or $10$, the sign, Wilcoxon, and van der Waerden scores, both homoscedastic and heteroscedastic linear models, and four error distributions, including the multivariate uniform, normal and Student $t_7$ ones.

Set $\tD = \mathrm{diag}\{1,2,\dots,m\}$ and $\tB=\mathbf{1}_m\mathbf{1}_p^\top$.
The observations $(\tY_i^\top,\tX_i^\top)^\top$, $i=1\dots,n$, are generated independently from model \eqref{RegMod1} with $\tX = (1,\tZ^\top)^\top$, where $\tZ$ is independent of $\teps$ and
\begin{itemize}[nosep]
\item[a)] Model A: $\tZ \sim \Ncal_{p-1}(\bs{0},\tI_{p-1})$, $\td=(1,0,\dots,0)^\top$ and $\teps \sim \Ncal_m(\bs{0},\tD^2)$,  
\item[b)] Model B: $\tZ \sim \Ncal_{p-1}(\bs{0},\tI_{p-1})$,  $\td=(1,0,\dots,0)^\top$ and  $\teps = \tD (\boldsymbol{\zeta} -\E \boldsymbol{\zeta})$ where $\boldsymbol{\zeta}$ is uniformly distributed on $[0,1]^m$, 
\item[c)] Model C: $\tZ \sim \Ncal_{p-1}(\bs{0},\tI_{p-1})$,  $\td=(1,0,\dots,0)^\top$ and  $\teps =  \tD \boldsymbol{\zeta}$ where $\boldsymbol{\zeta}$ follows the $m$-variate $t_7$ distribution (with 7 degrees of freedom), 
\item[d)] Model D:  $\tZ$ is uniformly distributed on $[0,1]^{p-1}$,  $\td= (1,2,\dots,p)^\top$ and $\teps \sim \Ncal_m(\bs{0},\tD^2)$.
\end{itemize}

Models A--C are homoscedastic, Model D is heteroscedastic. In all of the four models, the error term $\teps$ is axially symmetric around the axis $\tu_0 = (1,0,\dots,0)^\top$. Indeed, for this $\tu_0$, one can set $\tQ= \tI_m$ and $\Gu$ can be taken as $\Gu^\top = (\boldsymbol{0}|\tI_{m-1})$. If one partitions~$\teps$ as $\teps=(\eps_1,\teps_2^\top)^\top$, then $\teps$ is axially symmetric around $\tu_0$ if and only if $\teps$ has the same distribution as $(\eps_1,-\teps_2^\top)^\top$, see Lemma \ref{lemEqiv1}. This is clearly satisfied for all four models A--D.

The test statistic $T_n$ is computed from \eqref{test}. For the homoscedastic models A--C, the estimator  $\widehat{\tSigma}$  is  computed from \eqref{eq:Sigma.est} where $\widehat{F}_e$ is the empirical distribution function of $\widehat{W}_i$'s that are computed by means of the OLS estimator $\widehat{\tB}$. Model D satisfies the conditions of Theorem~\ref{th2}, see Lemma~\ref{lem:Norm}. Therefore, we can avoid the estimation of $\widehat{\td}$ by using the simplified estimator $\widehat{\tSigma}$ in \eqref{eq:Sigma}.

For the data generated from models A--D, we test the hypotheses $H_0(\tu_\alpha)$ with $\tu_{\alpha} = (\cos \alpha, \sin \alpha, \boldsymbol{0}_{m-2}^\top)^\top$ for various $\alpha \in [0,\pi/12]$. Such a null hypothesis holds only for $\alpha=0$. The empirical powers are 
computed from $N=1\,000$ replications and presented in Figures~\ref{fig1}--\ref{fig3} in \ref{ap2}.

Figure~\ref{fig1} compares the empirical power of  $T_n$ with normal scores and the test level $0.05$ for various dimensions $m\in\{2,5,10\}$, $p\in\{1,10\}$ and the sample size $n=1\,000$. Note that $p=1$ corresponds to the non-regression case (i.e., without stochastic regressors).  
It is apparent that the difference in the test performance is negligible for the two different regressor dimensions $p=1$ and $p=10$. 
In contrast, the power significantly declines with dimension $m$.

Figure~\ref{fig2} compares the performance of  $T_n$ with normal scores for different sample sizes $n\in \{250, 500, 1\;000\}$, $p=5$ and $m\in\{2,5,10\}$. It demonstrates that the test power increases with $n$ and $\alpha$ and decreases with $m$.
 
Finally, Figure~\ref{fig3} illustrates the impact of different (sign, Wilcoxon, and normal) score functions with regressor dimension $p=3$,
response dimension $m\in\{2,5,10\}$ and sample size $n=1\,000$. 
For models A, C and D with elliptical innovations, the normal scores and the Wilcoxon scores yield comparable power but the test using the sign scores is somewhat less powerful. In model B with the non-elliptical noise, the differences among the scores regarding the test power are more pronounced: the normal scores outperform the Wilcoxon scores, and the sign scores perform the worst.

The results confirm that the axial symmetry test based on $T_n$ is sized adequately and behaves quite well in all the settings considered. 
The test is quite fast and easy to compute, with a simple null asymptotic distribution. Therefore,  it is suitable even for large dimensions $m$ and $p$.   
The normal scores (obtained from the van der Waerden score function) seem to perform the best in most of the situations. The Wilcoxon and sign scores  can be recommended only in some special cases such as for heavy-tailed distributions or in the presence of extreme outliers.

\section{Real Data}
\label{Real}

The proposed test is illustrated with a data set on expenditures in different commodity groups from R package \textsf{HSAUR2}, \cite{hsaur}. The data come from a large survey on household expenditure in Hong Kong. 
The original expenditures of single persons were expressed in proportions of the total expenditures and logarithmized. The model \eqref{RegMod1} is considered for $\tY$ being a bivariate vector of the transformed expenditures for\\[1ex]
(FG)  food and goods, or \\
(GS) goods and services,\\[1ex]
and for $\tX = (1,G)^\top$ where $G$ is the gender of the customers ($G=1$ for men and $G=0$ for women). 
We will test the hypothesis of the conditional axial symmetry around an axis in direction $\tu_e = \frac{1}{\sqrt{2}}(1,1)^\top$, 
which is here equivalent to the test of the conditional exchangeability, see Lemma~\ref{lem:exchange}.
In this context, $H_0(\tu_e)$ means that the unexplained fluctuations in the two types of expenditures are statistically indistinguishable after controlling for the gender predictor. 
If the null hypothesis holds, the covariance structure of the error term admits a simplified representation, which may in turn yield more efficient inference.
Note that Theorem~\ref{prop:p1} assumes that $\tX = (1,\widetilde{\tX}^\top)^\top$ where   $\widetilde{\tX}$ has an absolutely continuous distribution, because \cite{Ha10a} define the directional quantiles  only under this assumption. 
However, after rereading all the relevant proofs, we believe that our Theorem 1 holds even for discrete  $\widetilde{\tX}$, as in
the present case. This is also confirmed by our numerical results.

The elements of matrix $\tB$ from the considered model \eqref{RegMod1} were estimated by OLS, and the variance parameter $d$ in $\td=(1,d)^\top$ was estimated from the sample variances of the OLS residuals  using two possible estimators $\widehat{d}$ and $\widetilde{d}$ that are described in more detail in  \ref{ap2}. 
  The test statistic is computed as in \eqref{test}, with $\widehat{\tSigma}$ from \eqref{eq:Sigma.est} with normal, Wilcoxon, and sign scores. The obtained p-values  are provided in Table~\ref{tab:1}.
For food and goods expenses, the null hypothesis $H_0(\tu_e)$ cannot be rejected at level 0.05 for all of the considered scores, but the result is rather borderline for the normal score function. Note that the results do not seem to be very affected by the choice of the estimator of the variance parameter $d$. In contrast, if the same testing procedure is applied to expenditures on goods and services, then $H_0(\tu_e)$ is rejected with p-value less than $0.001$ for all the considered score functions and both variance estimators. 
The conclusions of the test are in agreement with the plots of residuals $\widehat{\eps}_i = (\widehat{\td}^\top \tX)^{-1} (\tY_i - \widehat{\tB}\tX_i)$ shown in Figure~\ref{fig:data} in \ref{ap2}.

\begin{table}[ht]
\centering
\begin{tabular}{r|rr|rr}
\toprule
&\multicolumn{2}{c|}{Food \& Goods}&\multicolumn{2}{c}{Goods \& Services}\\
Scores & $\widehat{d}=  0.0385$ &  $\widetilde{d} = 0.0558 $ &   $\widehat{d} =  0.3815$ &  $\widetilde{d} =  0.2077$ \\ 
\midrule
Normal & 0.0848 & 0.0846 & 0.0002 & 0.0001 \\ 
Wilcoxon  & 0.1071 & 0.1067 & $<$ 0.0001& $<$ 0.0001 \\ 
 Sign & 0.2317 & 0.2317 & $<$ 0.0001 & $<$ 0.0001 \\ 
\bottomrule
\end{tabular}
\caption{P-values of the axial symmetry test  in direction $\tu_e = \frac{1}{\sqrt{2}}(1,1)^\top$ for bivariate logarithmic normalized expenses for  (FG) food and goods and (GS) goods and services for the three different scores and two possible variance estimators.}\label{tab:1}
\end{table}

\section{Concluding Remarks}\label{CR} 
This article proposes a test for conditional (or, error) axial symmetry around a line in a pre-specified direction in a multivariate heteroscedastic linear regression setup. The test is constructed from integrated rank scores. It is fast and easy to compute even for high dimensional data.
It is also quite powerful, especially with the van der Waerden scores. The test was also illustrated with a simulation study and applied to real data on expenses. 
 
\section*{Acknowledgements}

The research of Miroslav \v{S}iman was supported by the Czech Science Foundation project GA24-10078S. 
The research of \v{S}\'{a}rka Hudecov\'{a} was supported by the Czech Science Foundation project 25-15844S.

\bibliographystyle{apalike} 
\bibliography{HSbib}

\clearpage 

\appendix

\section{Auxiliary results about axial symmetry}\label{ap1}

Recall that a random vector $\tZ$ with a finite expectation is axially symmetric around an axis in direction $\tu$, $\|\tu\|=1$, 
if and only if $\tZ-\E \tZ$ has the same distribution as $\tR_\tu(\tZ-\E \tZ )$ for the orthogonal reflection matrix $\tR_\tu =2\tu\tu^\top  - \tI$.
  
Lemma \ref{lemEqiv1} provides a useful equivalent characterization of axial symmetry of a centered random vector. Note that the result has been implicitly used already in \cite{Hu20}.
 
\begin{lemma}\label{lemEqiv1}
Let $\teps$ be a random vector with $\E \teps=\boldsymbol{0}$. Define $\widetilde{\tQ} = (\tu|-\Gu)$.
Then $H_0(\tu)$ holds, i.e., $\Lcal(\teps) = \Lcal(\tR_\tu \teps)$, if and only if $\Lcal(\tQ^\top\teps) = \Lcal(\widetilde{\tQ}^\top\teps)$.
\end{lemma}

\begin{proof}
The definition of the orthogonal matrix $\tQ$ implies $\tu\tu^\top +\Gu\Gu^\top=\tI$. Therefore, $\tR_\tu = \tu\tu^\top-\Gu\Gu^\top=\tQ\widetilde{\tQ}^\top$.
Consequently, 
\[
\Lcal(\tQ^\top\teps) = \Lcal(\widetilde{\tQ}^\top\teps) \Rightarrow \Lcal(\tQ\tQ^\top\teps) = \Lcal(\tQ\widetilde{\tQ}^\top\teps) \Rightarrow
\Lcal(\teps) = \Lcal(\tR_\tu \teps)
\]
and 
\[
\Lcal(\teps) = \Lcal(\tR_\tu\teps) \Rightarrow \Lcal(\tQ^\top\teps) = \Lcal(\tQ^\top\tR_\tu \teps) = \Lcal(\widetilde{\tQ}^\top\teps).
\]
\end{proof}

Lemma \ref{lemEqiv2} shows that the eigenvectors of the (finite) variance matrix determine the directions of the axes around which the axial symmetry is possible. 

\begin{lemma}\label{lemEqiv2}
Let $\teps$ be a random vector with finite second order moments, $\E \teps=\boldsymbol{0}$ and $\var \teps = \tOmega$.
If $H_0(\tu)$ holds, then $\tu$ is an eigenvector of $\tOmega$.
\end{lemma}
\begin{proof}
Consider the same $\tQ$ and $\widetilde{\tQ}$ as in Lemma \ref{lemEqiv1}. Lemma \ref{lemEqiv1} implies that 
the vectors $\tQ^\top \teps$ and $\widetilde{\tQ}^\top\teps$ have the same variance matrix. Consequently, 
\[
\tQ^\top \tOmega \tQ = \widetilde{\tQ}^\top \tOmega \widetilde{\tQ}, 
\]
which is equivalent to 
\[
\tu^\top\tOmega\Gu=\boldsymbol{0}.
\]
It follows from the construction of $\Gu$ that  $\tv^\top\Gu = \boldsymbol{0}$ for a vector $\tv\in\R^m$ if and only if 
$\tv = \kappa \tu$ for some $\kappa\in\R$. This means that $\tOmega \tu = \kappa \tu$ for some $\kappa\in\R$. Therefore, 
$\tu$ is an eigenvector of $\tOmega$ corresponding to the eigenvalue $\kappa$. 
\end{proof}

\begin{lemma}\label{lem:Norm}
Let $\teps$ be a random vector with finite second order moments, $\E \teps=\boldsymbol{0}$ and $\var \teps = \tOmega$. If  $H_0(\tu)$ holds, then  $\tu^\top\teps$ and $\Gu^\top\teps$ are uncorrelated. Moreover, if 
 $\teps$ has a normal distribution, then $\tu^\top\teps$ and $\Gu^\top\teps$ are independent.  
\end{lemma}
 
\begin{proof}
Define $Z_1=\tu^\top\teps$ and $Z_2= \Gu^\top\teps$. If $H_0(\tu)$ holds, then, according to Lemma~\ref{lemEqiv2}, $\tu$~must be an eigenvector of $\tOmega$ that is associated with some eigenvalue $\lambda\in\R$.

The covariance matrix of  $Z_1$ and $\tZ_2$  is
\[
\cov(Z_1,\tZ_2) = \tu^\top\tOmega \Gu = (\tOmega \tu)^\top\Gu = \lambda \tu^\top\Gu = \boldsymbol{0}.
\]
If $\teps$ is normally distributed, then the joint distribution of $Z_1$ and $\tZ_2$ is also multivariate normal. The zero covariance matrix 
of $Z_1$ and $\tZ_2$ then implies their independence. 
\end{proof}
 
Next we are going to explore how axial symmetry relates to spherical symmetry, elliptical symmetry and exchangeability.  
Recall that a random vector $\tZ$ is spherically symmetric around $\bs{0}$ if $\mathcal{L}(\tZ) = \mathcal{L}(\tA \tZ)$ for any orthogonal matrix $\tA$. A random vector $\tW$ has an elliptical distribution with location $\tmu$ and positive semidefinite scatter matrix~$\tOmega$ if $\tW = \tmu+\tOmega^{1/2} \tZ$ for some $\tZ$ spherically symmetric around $\bs{0}$. 
If this $\tZ$ has finite second order moments, then  
 $\E \tW = \tmu$  and $\var \tW = \kappa \tOmega$ for certain $\kappa >0$. This shows that the scatter matrix of an elliptical distribution is specified uniquely up to a~multiplicative constant.

The distribution of a random vector $\tZ$ is said to be exchangeable if $\mathcal{L}(\tZ) = \mathcal{L}(\tP \tZ)$ for any permutation matrix $\tP$,
i.e., for any matrix $\tP$ obtained by permuting the rows of the identity matrix.  An elliptical random vector is exchangeable if and only if its location vector has identical components and its scatter matrix is compound symmetric.

\begin{lemma}\label{lem:spherical}
If $\teps$ is a spherically symmetric $m$-dimensional random vector with an arbitrary finite mean, then $\teps$ is axially symmetric around axes in direction $\tv$ for all $\tv\in\R^m$. 
\end{lemma}

\begin{proof}
Let $\tv\in\R^m$. The matrix $\tR_{\tv} =2\tv\tv^\top  - \tI$ is orthogonal. The spherical symmetry of $\teps$ then implies $\mathcal{L}(\teps-\E\teps) = \mathcal{L}(\tR_{\tv} (\teps-\E\teps))$. Therefore, $\teps$ is symmetric around an axis in direction $\tv$. 
\end{proof}

Note that Lemma \ref{lem:spherical} holds without the assumption of a finite variance matrix. Lemma~\ref{lem:El} about elliptical distributions 
avoids that assumption as well. 

\begin{lemma}\label{lem:El}
Let $\teps$ have an elliptical distribution with $\E \teps=\boldsymbol{0}$ and with a positive definite scatter matrix~$\tOmega$. Then 
$H_0(\tu)$ holds if and only if $\tu$ is an eigenvector of $\tOmega$.
\end{lemma}
\begin{proof}
The random vector $\teps$ can be expressed as $\teps = \tOmega^{1/2} \teta$, where $\teta$ has a spherically symmetric distribution. 
Consider spectral decomposition $\tOmega = \tV \boldsymbol{\Lambda} \tV^\top$ where $\boldsymbol{\Lambda}=\mathrm{diag}\{\lambda_1,\dots,\lambda_m\}$, $\tV = (\tv_1, \dots , \tv_m)$ is orthogonal,  $\lambda_1,\dots,\lambda_m$ are the eigenvalues of $\tOmega$ and $\tv_1,\dots,\tv_m$ are corresponding eigenvectors.

Assume $\tu=\tv_i$ for some $i\in\{1,\dots,m\}$ and note that 
\begin{equation*}
(2\tv_i \tv_i^\top -\tI) \tV = (-\tv_1,\dots , -\tv_{i-1},\tv_i,-\tv_{i+1},\dots, -\tv_m) = \widetilde{\tV}
\end{equation*}
where $\widetilde{\tV}$ is also an orthogonal matrix whose columns are eigenvectors of $\tOmega$ corresponding to $\lambda_1,\dots,\lambda_m$. 
Then for $\tR_\tu = 2\tu \tu^\top - \tI = 2\tv_i \tv_i^\top -\tI$, 
\[
\Lcal(\tR_\tu\teps) = \Lcal\big((2\tv_i \tv_i^\top -\tI) \tV \boldsymbol{\Lambda}^{1/2} \tV \teta\big)=
\widetilde{\tV} \boldsymbol{\Lambda}^{1/2} \Lcal\big(\tV \teta\big) = \widetilde{\tV} \boldsymbol{\Lambda}^{1/2} \Lcal\big(\widetilde{\tV} \teta\big) = \Lcal(\tOmega^{1/2}\teta) = \Lcal(\teps)
\]
because spherical symmetry of $\Lcal(\teta)$ implies $\Lcal(\teta) = \Lcal(\tV \teta) =\Lcal(\widetilde{\tV}\teta)$. 
This proves that  $H_0(\tu)$ holds.

Now assume that $H_0(\tu)$ holds for some $\tu\in\R^m$ such that $\|\tu\|=1$ and $\tR_{\tu} = 2\tu \tu^\top - \tI$. 
Consequently,  
\[
\Lcal(\teps) = \Lcal(\tOmega^{1/2}\teta) =  \Lcal(\tR_\tu\tOmega^{1/2}\teta) = \Lcal(\tR_\tu \tOmega^{1/2} \tR_\tu \teta),
\]
because $\teta$ is spherically symmetric and $\Lcal(\teta) = \Lcal(\tR_\tu \teta)$. 
This means that both $\tOmega^{1/2}\teta$ and $\tR_\tu \tOmega^{1/2} \tR_\tu \teta$ have the same elliptical distribution with the scale matrix $\tOmega$ that must satisfy
\[
\tOmega  = [\tR_\tu \tOmega^{1/2} \tR_\tu] [\tR_\tu \tOmega^{1/2} \tR_\tu]^\top = \tR_\tu \tOmega \tR_\tu,
\]
because $\tR_\tu^\top = \tR_\tu$ and $\tR_\tu \tR_\tu = \tI$. 
Since  $\tR_\tu \tu = \tu$, we get 
\[
\tOmega \tu =  \tR_\tu \tOmega \tR_\tu \tu = \tR_\tu \tOmega \tu,
\]
 and, therefore, 
  the vector $\tOmega \tu$ lies in the eigenspace of $\tR_\tu$ corresponding to its eigenvalue $1$. 
 Since $\tR_{\tu} \tv = -\tv$ for all $\tv$ such that $\tv^\top \tu =0$, it is straightforward to see that $\tR_{\tu}$ has two different eigenvalues $\mu_1=1$ (multiplicity 1) and $\mu_2 = -1$ (multiplicity $m-1$). Consequently, 
 the eigenspace corresponding to $\mu_1=1$ is equal to the span of $\tu$, and we get  $\tOmega \tu = \kappa \tu$ for some $\kappa\in\R$. This proves that $\tu$ is the eigenvector of $\tOmega$.
\end{proof}

Lemmas \ref{lem:spherical} and \ref{lem:El} somewhat extend Lemma \ref{lemEqiv2} and complement it by proving 
the other implication.

\begin{lemma}\label{lem:exchange}
Let $\teps$ be an $m$-dimensional   random vector with $\E \teps = \boldsymbol{0}$, and let $\tv = \frac{1}{\sqrt{m}}\boldsymbol{1}_m$. 
 \begin{enumerate}[label=(\roman*)]
\item If $m=2$, then $H_0(\tv)$ holds if and only if $\teps$ is exchangeable.  
\item Let $\teps$ have an elliptical distribution with a positive definite scatter matrix $\tOmega$. If $\teps$ is exchangeable, then 
$H_0(\tv)$ holds. Moreover, $H_0(\tu)$ is satisfied  for any unit vector $\tu\in\R^m$ orthogonal to $\tv$, i.e., for any $\tu$ such that $\tu^\top \boldsymbol{1}= \boldsymbol{0}$. 
\end{enumerate}
\end{lemma}

\begin{proof}
(i)  If $m=2$, then $\tv = (1,1)^\top/\sqrt{2}$ and $\tR_{\tv} = 2\tv \tv^\top -\tI = \big(\begin{smallmatrix}
  0& 1\\
  1 & 0
\end{smallmatrix}\big)$. In dimension $m=2$, there are only two permutation matrices: the identity matrix $\tI$ and $\tR_{\tv}$. This means that 
 the definition of axial symmetry around an axis in direction $\tv$ coincides with the definition of exchangeability.

(ii) The exchangeability implies that $\tP \tOmega \tP^\top = \tOmega$. Consequently, the scatter matrix $\tOmega$ has the well-known exchangeable structure, so it is equal to a positive multiple of  the matrix $ [(1-\rho)\tI + \rho \boldsymbol{1}\boldsymbol{1}^\top]$ for some $\rho\in(- (m-1)^{-1},1)$. It is easy to check that 
both $\tv$ and any $\tu$ such that $\tu^\top \boldsymbol{1}= \boldsymbol{0}$ are eigenvectors of $\tOmega$.
The rest follows from Lemma~\ref{lem:El}. 
\end{proof}

Lemma~\ref{lem:exchange} shows that $H_0(\tv)$, $\tv= \frac{1}{\sqrt{m}}\boldsymbol{1}_m$, is equivalent to the null hypothesis of exchangeability in the bivariate case with $m=2$. The following example demonstrates that it does not hold for $m>2$.

\begin{example}
Let $\tU$ be uniformly distributed on $[-1,1]^m$. Then the coordinates of $\tU$ are uniformly distributed on $[-1,1]$ and independent. Therefore, 
$\tU$ is exchangeable. In contrast, $\tU$ is not axially symmetric around an axis in direction $\tv = \frac{1}{\sqrt{m}}\boldsymbol{1}_m$, because the random vector $\tW=\tR_{\tv} \tU$,  where $\tR_{\tv} = 2\tv \tv^\top -\tI$, is distributed uniformly on the set $\{\tx: \tx = \tR_{\tv} \ty, \ty\in[-1,1]^m\}$. The set 
differs from the original cube $[-1,1]^m$ unless $\tR_{\tv} $ is a signed permutation matrix, which happens solely for $m=2$. 
\end{example}

Lemma~\ref{lem:exchange} provides a tool for testing exchangeability by means of testing  $H_0(\tv)$ for $\tv = \frac{1}{\sqrt{m}}\boldsymbol{1}_m$. For $m=2$, if $H_0(\tv)$ is rejected, then the exchangeability is rejected as well. 
For $m>2$, this procedure still works for elliptical distributions.

\section{Proofs of the main results}\label{ap1B}

The next lemma complements the characterization of axial symmetry. It will be used bellow in the proofs.

\begin{lemma}\label{lem0}
Let $\teps$ from \eqref{RegMod1} have a continuous distribution on $\R^m$ with density $f_{\teps}$ and $\E \teps=\boldsymbol{0}$.
Assume that $H_0(\tu)$ holds, i.e., $\Lcal(\teps) = \Lcal(\tR_\tu \teps)$. 
 \begin{enumerate}[label=(\roman*)]
 \item Let $W$ and $\tU$ be defined as in \eqref{eq:WU}. Then $\Lcal\bigl((W,\tU^\top)^\top|\tX\bigr) = \Lcal\bigl((W,-\tU^\top)^\top|\tX\bigr)$.
 \item 
 Define  $Z_1=\tu^\top\teps$ and $\tZ_2=\Gu^\top\teps$ and set $\tZ = (Z_1,\tZ_2^\top)^\top = \tQ^\top\teps$. 
 Then
 the vector~$\tZ$ has a continuous distribution on $\R^m$ with density 
  $f_{\tz}(z_1,\tz_2) = f_{\teps}(z_1 \tu+\Gu\tz_2)$ and $f_{\tz}(z_1,\tz_2) = f_{\tz}(z_1,-\tz_2)$ for all $z_1\in\R$ and $\tz_2\in\R^{m-1}$.
 \item The marginal density $f_{1}$ of $Z_1$ and the conditional density $f_{1|2}(z_1|\tz_2)$ of $Z_1$ given $\tZ_2$ satisfy
\begin{equation}\label{eq:lem1Z}
\E f_{1|2}(a|\tZ_2) = f_1(a)\ \ \text{and}\ \
\E f_{1|2}(a|\tZ_2)\tZ_2 =  \boldsymbol{0} 
\end{equation}
for any constant $a\in\R$.
 \end{enumerate}
 \end{lemma}
 
 \begin{proof}
The first claim follows directly from Lemma \ref{lemEqiv1}. The second statement results from Lemma \ref{lemEqiv1} and from the transformation 
theorem given the fact that $\teps = \tQ \tZ$.
Now consider the last statement. Obviously, 
\[
\E f_{1|2}(a|\tZ_2) = \int_{\R^{m-1}} f_{1|2}(a|\tz_2) f_2(\tz_2) \dd \tz_2 = \int_{\R^{m-1}}f_{\tz}(a,\tz_2) \dd \tz_2 =  f_1(a).
\]
Furthermore, it follows from (ii) that $f_{\tz}(a, \tz_2)\tz_2$ is odd and, therefore, 
 \begin{equation*}
\E f_{1|2}(a|\tZ_2)\tZ_2 = \int_{\R^{m-1}} f_{1|2}(a|\tz_2) \tz_2 f_2(\tz_2)\dd \tz_2  = \int_{\R^{m-1}}f_{\tz}(a,\tz_2)\tz_2\dd  \tz_2=\boldsymbol{0}.
\end{equation*}
\end{proof}

Recall that the vector $\th$ by definition minimizes \eqref{optim}. The next Lemma~\ref{lemHS} originates in \cite{Hu19},
where its statement is proved using geometrical ideas in the non-regression case ($p=1$) and by appealing to an analogy in the regression 
case ($p>1$). Here we provide a detailed formal proof for the general case with $p\geq 1$ to dispel any doubt about the validity of the claims.

\begin{lemma}\label{lemHS}
Let $H_0(\tu)$ and Assumptions \ref{as:phi} and \ref{as:XY} hold and  $\td^\top \tX>0$. Then $\tct=\bs{0}$ for all $\tau\in(0,1)$ and  $\ta_{\tau,\tu}^\top\tX$ is  equal to the conditional $\tau$-quantile 
of $\tu^\top\tY$ given $\tX$,
\[
\ta_{\tau,\tu}^\top\tX = F_{\tu^\top \tY|\tX}^{-1}(\tau|\tX) = \tu^\top\tB \tX + (\td^\top\tX) F^{-1}_{\tu^\top\teps}(\tau),
\]
where $F_{\tu^\top\teps}$ is the cumulative distribution function of $\tu^\top\teps$. 
\end{lemma}

\begin{proof}
The conditional distribution function of $\tu^\top \tY$ given $\tX$ satisfies for any $y\in\R$ and any $\tx \in \R^p$ such that $\td^\top\tx>0$
\begin{align*}
F_{\tu^\top \tY|\tX}(y|\tx) &= \P(\tu^\top \tY\leq y |\tX=\tx)= \P(\tu^\top\tB \tx + (\td^\top \tx) \tu^\top \teps \leq y) \\
&= \P\Big((\td^\top \tx) \tu^\top \teps \leq y - \tu^\top\tB \tx\Big) = F_{\tu^\top \teps}\left((y - \tu^\top\tB \tx)/{(\td^\top \tx) }\right),
\end{align*}
where  $F_{\tu^\top\teps}$ is the cumulative distribution function  of $\tu^\top\teps$. 
Consequently, the conditional quantile $F_{\tu^\top \tY|\tX}^{-1}(\tau|\tX)$  of $\tu^\top\tY$ given $\tX$ for $\tau\in(0,1)$ is
\[
F_{\tu^\top \tY|\tX}^{-1}(\tau|\tX) =  \tu^\top\tB \tX +  (\td^\top\tX)  F^{-1}_{\tu^\top\teps}(\tau) =  \Big(\tu^\top\tB +  F^{-1}_{\tu^\top\teps}(\tau) \td^\top\Big) \tX  =\tr_{\tau,\tu}^\top\tX
\]
for $ \tr_{\tau,\tu} = \tB^\top\tu +   F^{-1}_{\tu^\top\teps}(\tau) \,\td \in \R^p$. Note that 
\begin{equation}\label{eq:f}
\P(\tu^\top\tY - \tr_{\tau,\tu}^\top\tX \geq 0|\tX) = 1-F_{\tu^\top\tY|\tX}(\tr_{\tau,\tu}^\top \tX|\tX) =  1-\tau
\end{equation}
for any $\tau\in(0,1)$.

If the joint distribution of $(\bs{Y}^\top,\bs{X}^\top)^\top$ has a finite expectation and if it is continuous on a connected support, then its directional regression quantile vector $\th$ of \cite{Ha10a} is uniquely determined by the following necessary and sufficient conditions \citep{Ha10a}: 
\begin{align}
\bs{0} &= \frac{1}{1-\tau} \E\bigl\{\I(\tb_{\tau,\tu}^\top\bs{Y}-{\ta}_{\tau,\tu}^\top\bs{X}\geq 0)\bs{X}\bigr\}  
   -\frac{1}{\tau} \E\bigl\{\I(\tb_{\tau,\tu}^\top\bs{Y}-\ta_{\tau,\tu}^\top\bs{X}< 0)\bs{X}\bigr\}\label{dLaJ},\\
\bs{D}_{\tau,\tu} &= \frac{1}{1-\tau} \E\bigl\{\I({\tb}_{\tau,\tu}^\top\bs{Y}-\ta_{\tau,\tu}^\top\bs{X}\geq 0)\bs{Y}\bigr\} 
   -\frac{1}{\tau} \E\bigl\{\I(\tb_{\tau,\tu}^\top\bs{Y}-\ta_{\tau,\tu}^\top\bs{X}< 0)\bs{Y}\bigr\}\label{dLbJ},\\
1&=\tb_{\tau,\tu}^\top\tu,\label{dLlJ}
\end{align}
where $\bs{D}_{\tau,\tu} = \lambda_{\tau,\tu}\tu/\{\tau(1-\tau)\}$ and $\lambda_{\tau,\tu}$ is the Lagrange multiplier associated with (\ref{dLlJ}). 
In fact, \eqref{dLaJ}--\eqref{dLlJ} imply $\lambda_{\tau,\tu}= \E \rho_\tau \Bigl(\tb_{\tau,\tu}^\top\tY-{\ta}_{\tau,\tu}^\top\tX\Bigr)>0$.
Here, 
$\tb_{\tau,\tu} = \tu - \Gu \tc_{\tau,\tu}$ and $\tc_{\tau,\tu}= - \Gu^\top \tb_{\tau,\tu}$. We are going to show that the conditions \eqref{dLaJ}--\eqref{dLlJ} are satisfied for $(\ta_{\tau,\tu}^\top,\tb_{\tau,\tu}^\top)^\top=(\tr_{\tau,\tu}^\top, \tu^\top)^\top$. 
 
Observe that $\bs{b}_{\tau,\tu} = \tu$ always satisfies \eqref{dLlJ}. Therefore, it suffices to show \eqref{dLaJ} and \eqref{dLbJ}. 
The condition \eqref{dLaJ} can be rewritten as 
\begin{equation}\label{A-eqv}
\bs{0} = \E\bigl\{\I(\tb_{\tau,\tu}^\top\bs{Y}-\ta_{\tau,\tu}^\top\bs{X}\geq 0)\bs{X}\bigr\} - (1-\tau) \E \bs{X}.
\end{equation}
Similarly, \eqref{dLbJ} is equivalent to 
\begin{equation}\label{B-eqv}
 \lambda_{\tau,\tu} \tu  =  \E\bigl\{\I(\tb_{\tau,\tu}^\top\bs{Y}-\ta_{\tau,\tu}^\top\bs{X}\geq 0)\bs{Y}\bigr\} - (1-\tau) \E \bs{Y}.
\end{equation}
Let $\tau\in(0,1)$.  Then \eqref{eq:f} and the law of total expectation imply 
 \begin{align*}
\E\bigl\{\I(\tu^\top\bs{Y}-{\tr}_{\tau,\tu}^\top\bs{X}\geq 0)\bs{X}\bigr\} - (1-\tau) \E \bs{X} &= \E \left\{ \P(\tu^\top\bs{Y}-{\tr}_{\tau,\tu}^\top\bs{X}\geq 0|\tX) \tX- (1-\tau)\ \tX \right\} \\
&= \E\{ 0\cdot \tX\} = \boldsymbol{0}. 
 \end{align*}
This means that $(\ta_{\tau,\tu}^\top,\tb_{\tau,\tu}^\top)^\top=(\tr_{\tau,\tu}^\top, \tu^\top)^\top$ satisfies \eqref{A-eqv}   and consequently also \eqref{dLaJ}.

Recall that model \eqref{RegMod1} leads to $\E[\tY|\tX] = \tB \tX$. This and \eqref{eq:f} result in
 \begin{align*}
\E\big\{\I(\tu^\top\bs{Y}&-{\tr_{\tau,\tu}}^\top\bs{X}\geq 0)\bs{Y}\big\} - (1-\tau) \E \bs{Y} \\
&= 
\E \left\{ \E[\I( \tu^\top\bs{Y}-{\tr_{\tau,\tu}}^\top\bs{X}\geq 0)\bs{Y}|\tX] - (1-\tau) \E[\bs{Y}|\tX] \right\}\\
&= 
\E \left\{ \E[\I( \tu^\top\bs{Y}-{\tr_{\tau,\tu}}^\top\bs{X}\geq 0)\bs{Y}|\tX] -  \P( \tu^\top\bs{Y}-\tr_{\tau,\tu}^\top\bs{X}\geq 0)|\tX)  \tB\tX \right\}\\
&= 
\E \left\{ \E[\I( \tu^\top\bs{Y}-{\tr_{\tau,\tu}}^\top\bs{X}\geq 0)\bs{Y}|\tX] -  \E[\I(\tu^\top\bs{Y}-{\tr_{\tau,\tu}}^\top\bs{X}\geq 0) \tB\tX  |\tX] \right\}\\
&=
\E \left\{ \E[\I(\tu^\top\bs{Y}-{\tr_{\tau,\tu}}^\top\bs{X}\geq 0)(\bs{Y} - \tB \tX)|\tX] \right\}\\
& = \E \left\{ \E[\I(\tu^\top\teps - F_{\tu^\top\teps}^{-1}(\tau) \geq 0) (\td^\top\tX) \teps|\tX] \right\} \\
& = \td^\top\E[\tX] \cdot   \E\left[\I(\tu^\top\teps - F_{\tu^\top\teps}^{-1}(\tau) \geq 0)\teps\right], 
\end{align*}
where we used the definition of $\tr_{\tau,\tu}$, the model structure, the assumption $\td^\top \tX>0$, and the independence of 
$\teps$ and $\tX$.  Under $H_0(\tu)$, $\mathcal{L}(\teps) = \mathcal{L} (\tR_\tu \teps)$ for $\tR_\tu = 2\tu \tu^\top - \tI$.  
Define $q= F_{\tu^\top\teps}^{-1}(\tau)$. Then
\[
\E\{  \I(\tu^\top\teps - q \geq 0) \cdot \teps\}= \E \{ \I(\tu^\top (\tR_\tu \teps) - q \geq 0)\cdot \tR_\tu \teps\} =\tR_\tu  \E\{ \I(\tu^\top\teps - q \geq 0)\cdot   \teps\},
\]
because $\tu^\top\tR_\tu = \tu^\top$. Hence, $(\tI-\tR_\tu) \E\{  \I(\tu^\top\teps - q \geq 0) \cdot \teps\} = \boldsymbol{0}$. 
Note that $\tI - \tR_\tu = 2(\tI - \tu \tu^\top)$, where  $(\tI - \tu \tu^\top)$ is the projection matrix to the space orthogonal to $\tu$. This implies 
that $\E\{  \I(\tu^\top\teps - q \geq 0) \cdot \teps\}  = \kappa \tu$ for some $\kappa\in\R$.
Moreover, $\kappa= \kappa \tu^\top \tu = \E\{  \I(\tu^\top\teps - q \geq 0) \cdot \tu^\top\teps\}>0$,
because $\E (\tu^\top\teps)=0$. This proves that \eqref{B-eqv} and consequently also 
 \eqref{dLbJ} hold for $(\tr_{\tau,\tu}^\top, \tu^\top)^\top$. 
 
To sum up, the conditions \eqref{dLaJ}--\eqref{dLlJ} are satisfied for $(\tr_{\tau,\tu}^\top, \tu^\top)^\top$, and  the uniqueness of  the directional quantiles implies that $\tb_{\tau,\tu} = \tu$ for all $\tau\in(0,1)$. Therefore, $\tc_{\tau,\tu} = - \Gu^\top\tu = \boldsymbol{0}$ for all $\tau \in(0,1)$. The rest of the assertion follows easily from the definition of $\ta_{\tau,\tu}$.
\end{proof}

\begin{lemma} \label{lem1}  
 Let the assumptions of Theorem~\ref{prop:p1} be satisfied.  
For any $\tau\in(0,1)$ define
\begin{equation}\label{eq:Vn}
\what{\tV}_n(\tau):=
 \frac{1}{\sqrt{n}}\Gu^\top\ttY^\top \left(\ttI- \ttM_{\ttX})[\what{\ta}(\tau)-(1-\tau)\mathbf{1}\right]
\end{equation}  
and
\begin{equation}\label{eq:Vn0}
  \what{\tV}_n^0(\tau)= \frac{1}{\sqrt{n}}\sum_{i=1}^n \{\tau - \I[\tu^\top\tY_i<\ta_{\tau,\tu}^\top\tX_i]\}\{\Gu^\top\tY_i-\Gu^\top \ttB \tX_i\}.
 \end{equation}
Then, for any $\tau\in(0,1)$, 
\begin{equation*}
\what{\tV}_n(\tau) = \what{\tV}_n^0(\tau) +\tR_n(\tau),
\end{equation*}
where $\|\tR_n(\tau)\|\Pto 0$ for all $\tau\in(0,1)$ and the convergence is 
uniform for $\tau\in[\varepsilon,1-\varepsilon]$.
\end{lemma}

\begin{proof}
Under the stated conditions, \cite{Hu19} showed that 
\begin{equation*}
\what{\tV}_n(\tau)  =  \what{\tV}_n^1(\tau) + \tR_n(\tau),
\end{equation*}
where  $\tR_n(\tau)$ has the properties stated above and  
\begin{equation*}
 \what{\tV}_n^1(\tau) = \frac{1}{\sqrt{n}}\sum_{i=1}^n \left\{\tau - \I[\tu^\top\tY_i<\ta_{\tau,\tu}^\top\tX_i]\right\}\left\{\Gu^\top\tY_i-\Gu^\top \ttH(\tau)\ttG(\tau)^{-1} \tX_i\right\}
\end{equation*}
for matrices
\begin{equation*}
\ttG(\tau)= \E\{ \fuX(\ta_{\tau,\tu}^\top\tX|\tX)\tX\tX^\top\}\quad \text{and} \quad \ttH(\tau)=\E\{\fuXG(\ta_{\tau,\tu}^\top\tX|\tX, \Gu^\top\tY)\tY\tX^\top\},
\end{equation*}
where $\fuX$ is the density of $\tu^\top\tY$ given $\tX$, and $\fuXG$ is the density of $\tu^\top\tY$ given $(\tX^\top,\tY^\top\Gu)^\top$. 
Consequently, it suffices to show that  
\begin{equation*}
\Gu^\top\ttH(\tau)\ttG(\tau)^{-1} =\Gu^\top\ttB,\ \ \text{i.e.,}\ \ \Gu^\top\ttH(\tau) =\Gu^\top\ttB\ttG(\tau). 
\end{equation*}

Define $\tZ=(Z_1,\tZ_2^\top)^\top$, $f_1$ and $f_{1|2}$ as in Lemma~\ref{lem0}.  Then it follows from the considered  regression model \eqref{RegMod1} that
\begin{align*}
\fuX(\ta_{\tau,\tu}^\top\tX|\tX) &= f_1\left(\frac{\ta_{\tau,\tu}^\top\tX-\tu^\top \ttB\tX}{\td^\top\tX}\right)\cdot \frac{1}{(\td^\top\tX)}\ \ \text{and}\\
\fuXG(\ta_{\tau,\tu}^\top\tX|\tX, \Gu^\top\tY) &= f_{1|2}\left(\frac{\ta_{\tau,\tu}^\top\tX-\tu^\top\ttB\tX}{\td^\top\tX} \Big| \Gu^\top\teps \right)\cdot \frac{1}{(\td^\top\tX)}.
\end{align*}
Therefore,
\[
\ttG(\tau)=\E f_1\left(\frac{\ta_{\tau,\tu}^\top\tX-\tu^\top\ttB\tX}{\td^\top\tX}\right)\cdot \frac{1}{(\td^\top\tX)} \tX\tX^\top
\]
and
\[
\ttH(\tau)=  \underbrace{\E\{\fuXG(\ta_{\tau,\tu}^\top\tX|\tX, \Gu^\top\tY)\ttB\tX\tX^\top\} }_{\boldsymbol{I}_1}+\underbrace{\E\{\fuXG(\ta_{\tau,\tu}^\top\tX|\tX, \Gu^\top\tY)(\td^\top\tX)\teps\tX^\top\}}_{\boldsymbol{I}_2}.
\]
It follows from \eqref{eq:lem1Z} in Lemma~\ref{lem0} that
\begin{align*}
\boldsymbol{I}_1 & = \E  f_{1|2}\left(\frac{\ta_{\tau,\tu}^\top\tX-\tu^\top\ttB\tX}{\td^\top\tX} \Big| \Gu^\top\teps \right)\cdot \frac{1}{(\td^\top\tX)}\ttB\tX\tX^\top\\ 
& = \E \left\{\E\left[f_{1|2}\left(\frac{\ta_{\tau,\tu}^\top\tX-\tu^\top\ttB\tX}{\td^\top\tX} \Big| \Gu^\top\teps \right)|\tX\right] \frac{1}{(\td^\top\tX)} \ttB\tX\tX^\top \right\}\\
& =  \ttB \E\left\{ f_1\left(\frac{\ta_{\tau,\tu}^\top\tX-\tu^\top\ttB\tX}{\td^\top\tX} \right)  \frac{1}{(\td^\top\tX)} \tX\tX^\top\right\} = \ttB \ttG(\tau).
\end{align*}
Furthermore,
\[
\boldsymbol{I}_2  =  \E  f_{1|2}\left(\frac{\ta_{\tau,\tu}^\top\tX-\tu^\top\ttB\tX}{\td^\top\tX} \Big| \Gu^\top\teps \right) \teps\tX^\top
\]
and
\[
\Gu^\top \boldsymbol{I}_2  = \E\left\{ \E\left[ f_{1|2}\left(\frac{\ta_{\tau,\tu}^\top\tX-\tu^\top\ttB\tX}{\td^\top\tX} \Big| \Gu^\top\teps \right)\Gu^\top\teps \Big|\tX\right] \tX^\top\right\}=\boldsymbol{0}
\]
due to \eqref{eq:lem1Z}. 
To sum up,
\[
\Gu^\top \ttH(\tau) = \Gu^\top \ttB \ttG(\tau) + \boldsymbol{0},
\]
which concludes the proof.
\end{proof}
\medskip

\noindent \emph{Proof of Theorem \ref{prop:p1}.}
Since the first column of $\ttX$ is $\mathbf{1}$ and  $\ttM_{\ttX}$ is the projection matrix to the space generated by the columns of $\ttX$, it holds that $\ttM_{\ttX} \mathbf{1} = \mathbf{1}$ and, consequently,
\begin{equation}\label{eq:Mx}
 (\ttI-\ttM_{\ttX})\mathbf{1}=\mathbf{0}\ \ \text{and}\ \ (\ttI - \ttM_{\ttX}) \ah(\tau)=\ah(\tau)- (1-\tau)\mathbf{1}.
\end{equation}

Therefore,  $\what{\tV}_n$ from \eqref{eq:Vn} and the definition of $\widehat{\tb}$ imply
\begin{align*}
- \int_{0}^1 \phi(\tau) \dd  \what{\tV}_n(\tau) & = \frac{1}{\sqrt{n}}\Gu^\top\ttY^\top (\ttI_n- \ttM_{\ttX}) \left[ -\int_{0}^1\phi(\tau) \dd \what{\ta}(\tau) - \int_0^1  \phi(\tau) \dd \tau\right]\\ 
& = \frac{1}{\sqrt{n}}\Gu^\top\ttY^\top (\ttI_n- \ttM_{\ttX}) \bh = \what{\tS}_n.
\end{align*}

Lemma~\ref{lem1}  proves that  $\what{\tV}_n(\tau)  =  \what{\tV}_n^0(\tau) + \tR_n(\tau)$ for $\what{\tV}_n^0$ from \eqref{eq:Vn0}, where
$\|\tR_n(\tau)\|\Pto 0$ for all $\tau\in(0,1)$ and the convergence is uniform on $[\varepsilon,1-\varepsilon]$.
Furthermore, $\ta_{\tau,\tu}^\top\tX$ is the conditional quantile $F_{\tu^\top \tY|\tX}^{-1}(\tau)$ according to Lemma~\ref{lemHS}. Consequently,
\begin{align*}
\I\left[\tu^\top\tY_i<\ta_{\tau,\tu}^\top\tX_i\right]  &= \I\left[\tu^\top\tY_i<F_{\tu^\top \tY|\tX}^{-1}(\tau|\tX_i)\right] \\
&= 
 \I\left[\tu^\top\tB \tX_i + (\td^\top \tX_i) \tu^\top\teps_i< \tu^\top\tB \tX_i + (\td^\top \tX_i) F^{-1}_{\tu^\top\teps} (\tau) \right] \\
&=
\I\left[F_{\tu^\top \teps}(\tu^\top\teps_i)<\tau\right]
\end{align*}
almost surely for all $i=1,\dots,n$.

It follows from  Assumption \ref{as:phi} that $ \int_{0}^1 \phi(\tau) \dd \tR_n(\tau) \Pto \boldsymbol{0}$ as $n\to\infty$. 
Consequently, the test statistic $\what{\tS}_n$ can be expressed as
\begin{align*}
\what{\tS}_n & = \int_{0}^1 -\phi(\tau) \dd \what{\tV}_n(\tau) = \int_{0}^1 -\phi(\tau) \dd \what{\tV}_n^0(\tau) + o_{\P}(1)\\
& = \frac{1}{\sqrt{n}}\sum_{i=1}^n(\Gu^\top\tY_i-\Gu^\top \ttB \tX_i) \left[ \int_0^1 -\phi(\tau) \dd \tau +\int_0^1 \phi(\tau) \dd  \I[F_{\tu^\top\teps}(\tu^\top\teps_i)< \tau]  \right] +o_{\P}(1)\\
&  = \frac{1}{\sqrt{n}}\sum_{i=1}^n\tU_i  \left[\phi\left(F_{\tu^\top\teps}(\tu^\top\teps_i)\right) - \overline{\phi}\right] +o_{\P}(1) = \frac{1}{\sqrt{n}}\sum_{i=1}^n \txi_i +o_{\P}(1)
\end{align*} 
for independent and identically distributed random vectors 
$\txi_i:=\tU_i  \left[\phi\left({F_{\tu^\top\teps}(\tu^\top\teps_i)}\right) - \overline{\phi}\right]$, $i=1,\dots,n$. Furthermore, $F_{\tu^\top\teps}(\tu^\top\teps_i) = F_{W|X}(W_i|\tX_i)$, where $F_{W|X}(\cdot |\tx)$ is the cdf of $W$ given $\tX=\tx$.
It follows from {Lemma~\ref{lem0} (i) that 
\begin{align*}
\E \txi_i &= \E \tU_i \left[ \phi({F_{\tu^\top\teps}(\tu^\top\teps_i)})-\overline{\phi}\right] = \E\left\{ \E \left(\tU_i \left[ \phi(F_{W|\tX}(W_i|\tX_i))-\overline{\phi}\right]|\tX_i\right)\right\} \\
&= \E\left\{ \E \left((-\tU_i) \left[ \phi(F_{W|\tX}(W_i|\tX_i))-\overline{\phi}\right]|\tX_i\right)\right\} = -\E\txi_i, 
\end{align*}
which implies $\E \txi_i=\bs{0}$. 
Moreover, it follows from Assumptions \ref{as:phi} and \ref{as:XY} that
\begin{equation*}
\var \txi_i = \E \txi_i \txi_i^\top = \E \tU_i\tU_i^\top \left[ \phi(F_{\tu^\top\teps}(\tu^\top\teps_i))-\overline{\phi}\right]^2
\end{equation*}
is finite, and, therefore, the  assertion follows directly from the central limit theorem.
\qed

\noindent \emph{Proof of Theorem \ref{th2}.}
If $\phi$ is the sign score function, then $\overline{\phi} = 0$ and $[\phi(t)]^2=1$ for all $t\in[0,1] \smallsetminus \{0.5\}$}. Therefore, {$\sigma^2_{\phi}=1$ and $\tSigma = \E \tU_1\tU_1^\top$. 

If $\tU_1$ and $\tu^\top\teps_1$ are conditionally independent given $\tX_1$, then
\begin{align*}
\tSigma &= \E \left\{ \E\left[\tU_1\tU_1^\top \left[ \phi\big(F_{ \tu^\top \teps}(\tu^\top \teps_1)\big)-\overline{\phi}\right]^2\Big|\tX_1\right]\right\}\\
&=  \E \left\{ \E\left[\tU_1\tU_1^\top|\tX_1\right]\cdot \E\left[ \left[ \phi\big(F_{ \tu^\top \teps}(\tu^\top \teps_1)\big)-\overline{\phi}\right]^2\Big|\tX_1\right]\right\}\\
&=\left( \E  \tU_1 \tU_1^\top \right)\cdot  \E \left[ \phi(F_{ \tu^\top \teps}(\tu^\top \teps_1))-\overline{\phi}\right]^2 = \sigma^2_{\phi}  \E  \tU_1 \tU_1^\top,
\end{align*}
due to the independence of $\tX_1$ and $\teps_1$. \qed

\section{Additional results to the practical part}\label{ap2}

Figures \ref{fig1}--\ref{fig3} show the empirical power for the test of $H_0(\tu_{\alpha})$ for the vectors $\tu_{\alpha} = (\cos(\alpha),\sin(\alpha),\bs{0}_{m-2}^\top)^\top$,  $\alpha\in[0,\pi/12]$, obtained for models A--D considered in the Monte Carlo simulation study in Section~\ref{Demo}. 

\smallskip

In Section~\ref{Real}, we consider model \eqref{RegMod1} for $m=2$ and $p=2$ and $\tX = (1, G)^\top$, where $G$ is the gender indicator ($G=1$ for men) for $n$ independent realizations. We set $\td = (1,d)^\top$,  $\tB = (b_{ij})_{i,j=1}^2$, $\tY=(Y_1,Y_2)^\top$ and $\teps=(\eps_1,\eps_2)^\top$, so the generic model can be rewritten as
\begin{align*}
Y_{1} & = b_{11} + b_{12}G+ (1+d G)\eps_{1},\\
Y_{2}& = b_{21}+b_{22}G + (1+d G)\eps_{2}.
\end{align*}
Consequently, the parameter $d$ satisfies
\[
(1+d)^2 = \frac{\var[Y_{1}|G=1]}{\var[Y_{1}|G=0]} =\frac{\var[Y_{2}|G=1]}{\var[Y_{2}|G=0]}.
\]
At the same time,  we assume that $\td^\top \tX = 1+d \cdot G >0$, so $d>0$. 
Denote as $v_j = \var[Y_{j}|G=1]/\var[Y_{j}|G=0]$ the variance ratio for males and females for the $j$-th component of $\tY$, $j=1,2$. 
If $\widehat{v}_1$ is a consistent estimator of $v_1$, then a consistent estimator of $d$ is obtained  as
\begin{equation}\label{eq:d}
\widehat{d}_1= \sqrt{\widehat{v}_1}-1.
\end{equation}
At the same time, an analogous consistent estimator $\widehat{d}_{2}$ of $d$ can be obtained from a  consistent estimator $\widehat{v}_2$ of $v_2$. We then define our estimator of $d$, that combines information from both components of $\tY$, as
$
\widehat{d} = \frac{1}{2}(\widehat{d}_{1} + \widehat{d}_{2}).
$

The estimator $\widehat{v}_j$ can be obtained easily as a ratio of the sample variances of the $j$-th component of the OLS residuals. Let $\boldsymbol{e}_i = (e_{i,1},e_{i,2})^\top = \tY_i - \widehat{\tB} \tX_i$ be the OLS residual for the $i$-th subject, $i=1,\dots,n$.  Let $S^2_{M,j}$ and $S^2_{F,j}$ be the sample variances of  those $e_{1,j},\dots,e_{n,j}$ corresponding to males ($G=1$) and females ($G=0$), respectively, $j=1,2$. Then $\widehat{v}_j = S^2_{M,j}/ S^2_{F,j}$. Mathematically, 
let $K = \sum_{i=1}^n G_i$ be the number of males and assume that $G_1=\dots=G_K=1$ and $G_{K+1}=\dots=G_n = 0$ without loss of generality. Then for $j=1,2$, 
\begin{align*}
S^2_{M,j} &= \frac{1}{K-1}\sum_{i=1}^K \left[ e_{i,j} - \frac{1}{K}\sum_{r=1}^K e_{r,j}\right]^2,\\
 S^2_{F,j}&= \frac{1}{n-K-1}\sum_{i=K+1}^n \left[ e_{i,j} - \frac{1}{n-K}\sum_{r=K+1}^n e_{r,j}\right]^2.
\end{align*}

Alternatively, under the null hypothesis, $\teps$ is exchangeable, so $\var \eps_1 = \var \eps_2$ and $\var[Y_1|G] = \var[Y_2|G]$. Consequently,  
we can define an estimator $\widetilde{d}$ of $d$ as in \eqref{eq:d}, but with $\widehat{v}_1$ replaced by 
 a pooled estimator  $\widehat{v} = S^2_{M}/S^2_{F}$, where  $S^2_{M}$ and $S^2_{F}$ are sample variances 
  of all  OLS residuals  corresponding to males and females, respectively.  Here, $S^2_{M}$ is computed as the sample variance of $e_{1,1},\dots,e_{K,1},e_{1,2},\dots,e_{K,2}$ and similarly for $S^2_F$. 
  Mathematically, 
\begin{align*}
S^2_{M} &= \frac{1}{2K-1}\sum_{j=1}^2\sum_{i=1}^K\left[ e_{i,j} - \frac{1}{2K}\sum_{h=1}^2 \sum_{r=1}^K e_{r,h}\right]^2,\\
 S^2_{F}&= \frac{1}{2(n-K)-1}\sum_{j=1}^2\sum_{i=K+1}^n \left[ e_{i,j} - \frac{1}{2(n-K)}\sum_{h=1}^2 \sum_{r=K+1}^n e_{rh}\right]^2.
\end{align*}

\medskip

Figure~\ref{fig:data} shows the  standardized residuals $\widehat{\eps}_{i,j}= e_{i,j} / (1+\widehat{d} \cdot G_i)$, $i=1,\dots,n$, $j=1,2$,
computed from the regression model for transformed expenditures considered in Section~\ref{Real}. The left panel corresponds to  food and goods, where the test does not reject the null hypothesis of 
the symmetry around axis in direction $(1,1)^\top$, i.e.
exchangeability.   The right panel provides residuals for goods and services, where the same hypothesis is rejected with p-value less than $0.001$.  

\begin{figure}[htbp]
\begin{center}
 \includegraphics[width=\textwidth]{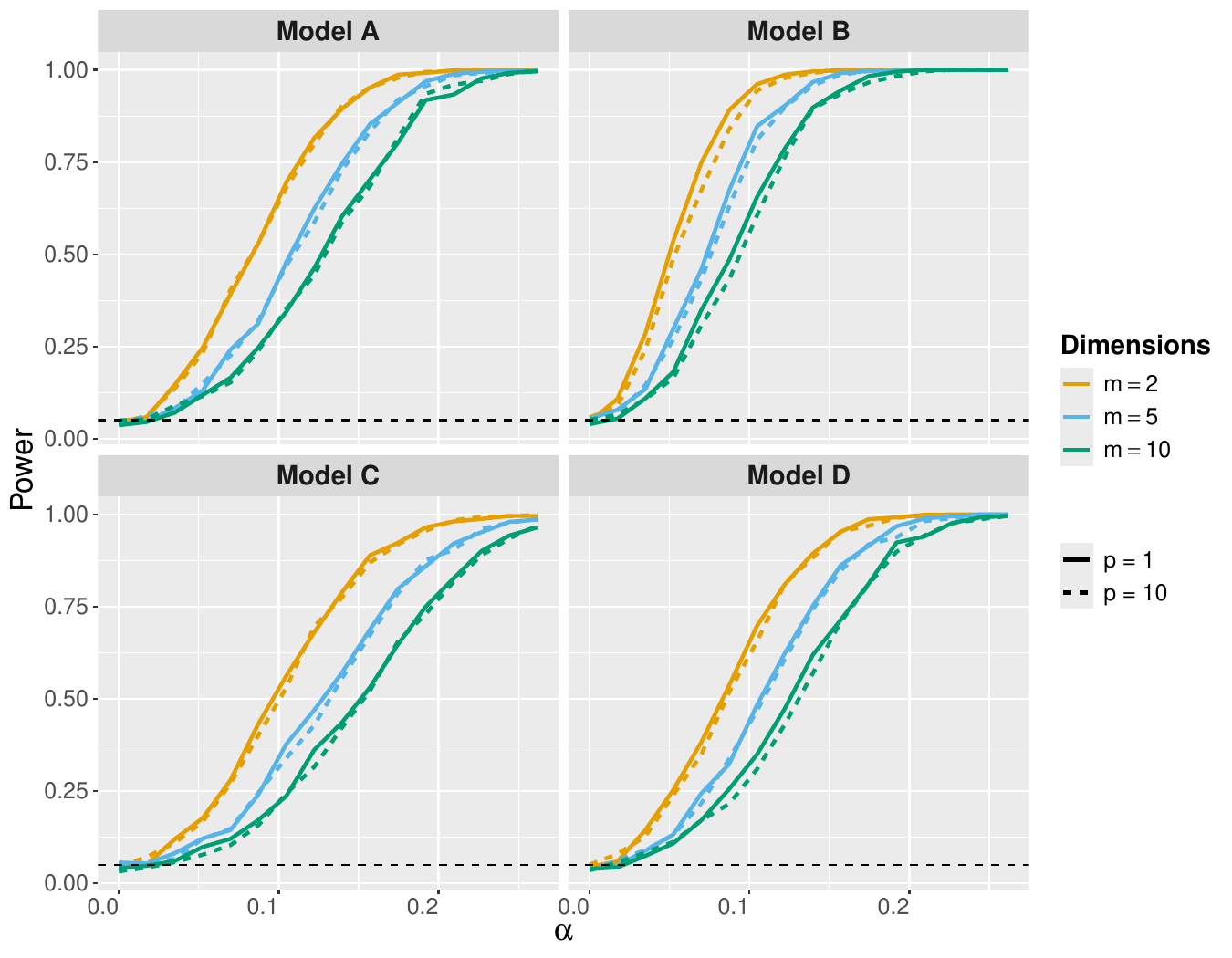}
 \end{center}
\caption{{\textbf{Power in the location and regression case}}. Empirical power of the test based on $T_n$ with normal scores for testing $H_0(\tu_{\alpha})$, $\tu_{\alpha} = (\cos(\alpha),\sin(\alpha),\bs{0}_{m-2}^\top)^\top$ for various values $\alpha\in[0,\pi/12]$, 
response dimensions $m\in\{2,5,10\}$, regressor dimensions $p\in\{1,10\}$, sample size $n=1\,000$ and models A--D. The test level $0.05$ is stressed with the dashed black line.}\label{fig1}
\end{figure}

\begin{figure}[htbp]
\begin{center}
\begin{center}
 \includegraphics[width=\textwidth]{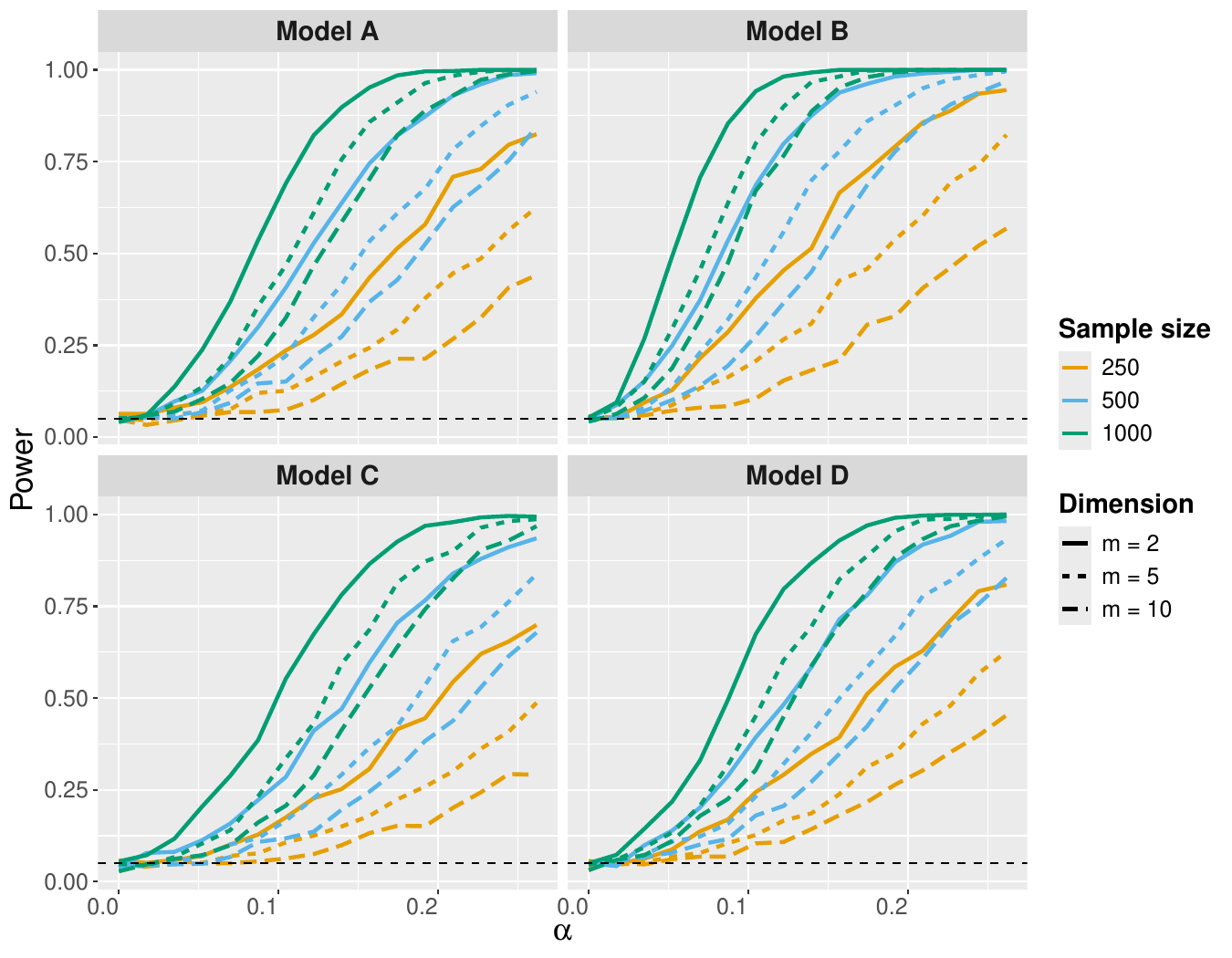}
 \end{center}
\end{center}
\caption{\textbf{Power for various sample sizes.} Empirical power of the test based on $T_n$ with normal scores for testing $H_0(\tu_{\alpha})$, $\tu_{\alpha}= (\cos(\alpha),\sin(\alpha),\bs{0}_{m-2}^\top)^\top$ for various values $\alpha\in[0,\pi/12]$, sample sizes $n\in\{250,500,1\,000\}$, 
response dimensions $m\in\{2,5,10\}$, regressor dimension $p=5$ and models A--D. The test level $0.05$ is stressed with the dashed black line. }\label{fig2}
\end{figure}

\begin{figure}[htbp]
\begin{center}
\begin{center}
 \includegraphics[width=\textwidth]{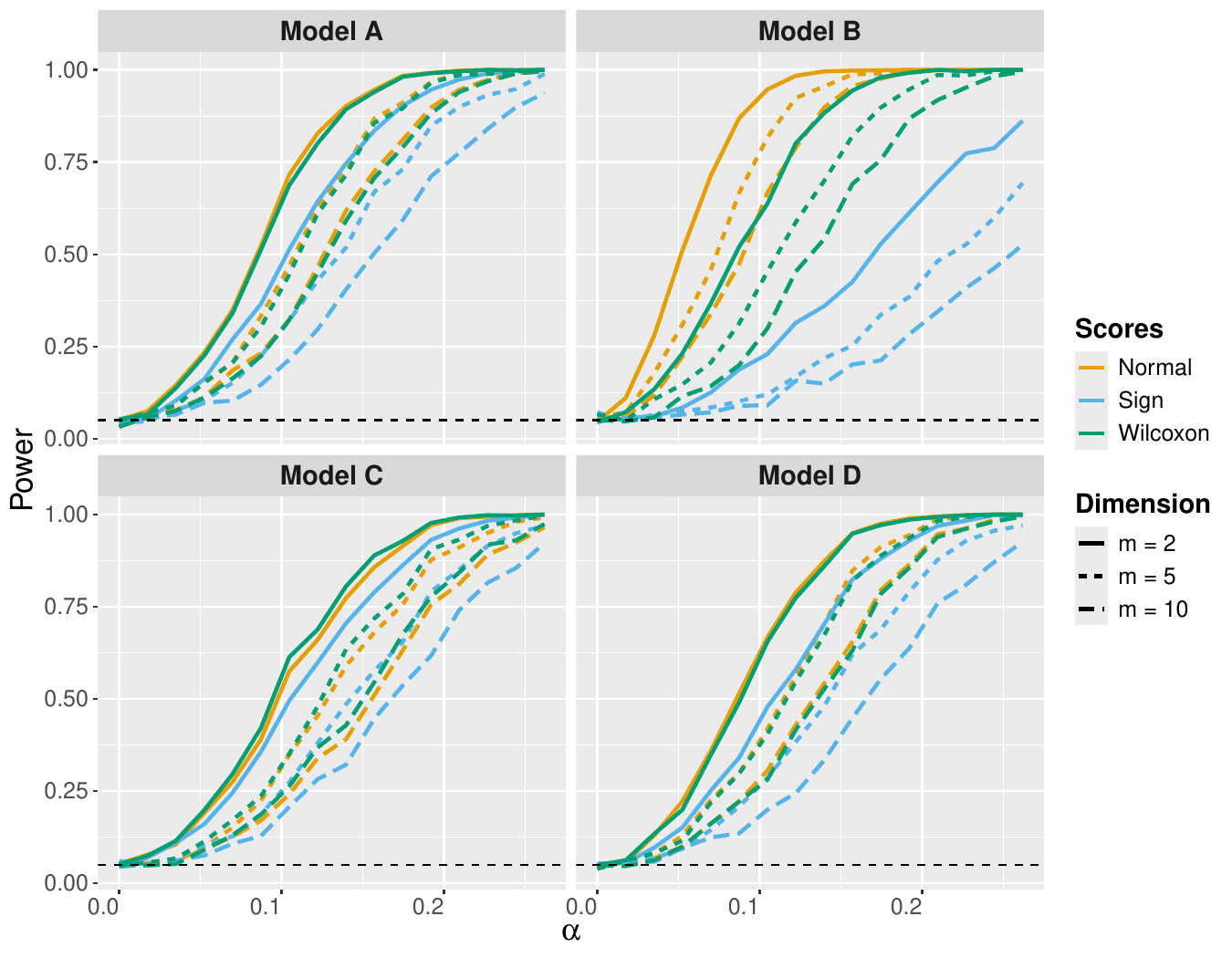}
 \end{center}
\end{center}
\caption{\textbf{Power for various scores.} Empirical power of the test based on $T_n$ for testing $H_0(\tu_{\alpha})$, $\tu_{\alpha}= (\cos(\alpha),\sin(\alpha),\bs{0}_{m-2}^\top)^\top$ for various values $\alpha\in[0,\pi/12]$, three different score functions, response dimensions $m\in\{2,5,10\}$, regressor dimension $p=3$, sample size $n=1\,000$ and models A--D.  The test level $0.05$ is stressed with the dashed black line. }\label{fig3}
\end{figure}

\begin{figure}[htbp]
\begin{center}
\begin{center}
 \includegraphics[width=\textwidth]{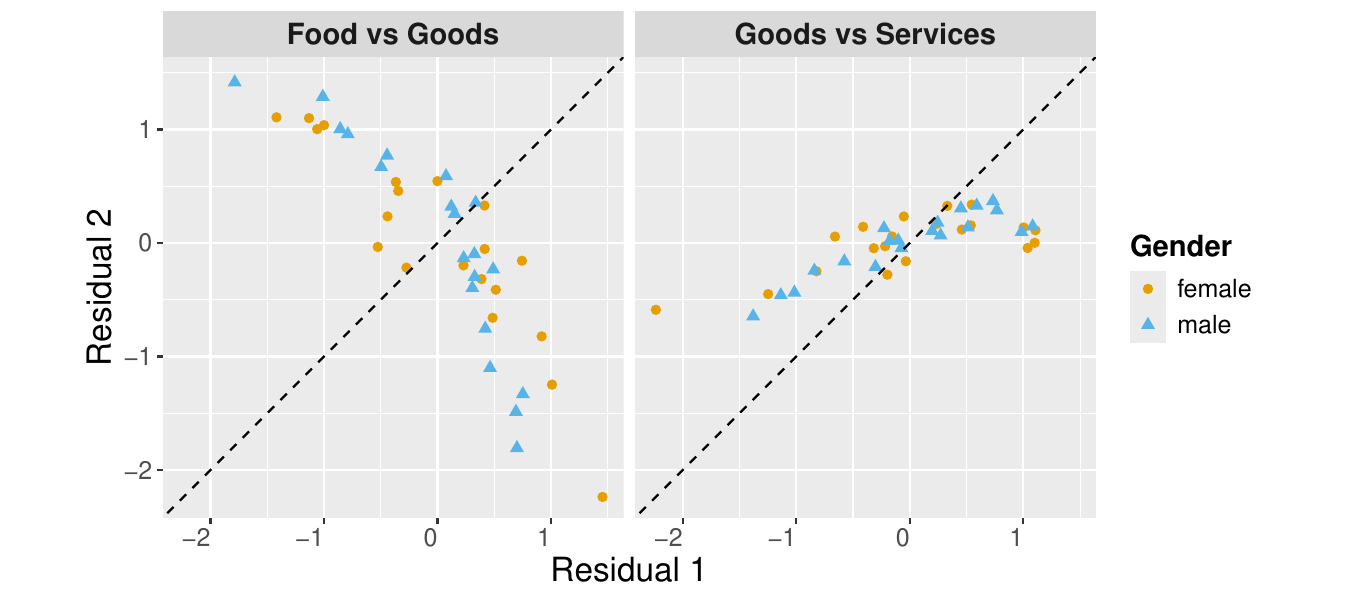}
 \end{center}
\end{center}
\caption{Residuals from the regression model for log-tranformed expenditures on food and goods (left panel) or goods and services (right panel). The straight dashed line is the axis in direction $\tu_e = \frac{1}{\sqrt{2}}(1,1)^\top$.}\label{fig:data}
\end{figure}

\end{document}